\documentclass[11pt]{amsart}

\usepackage{amsmath,amsthm,amssymb}

\usepackage{epic,eepic}
\usepackage[all]{xy}
\usepackage{tikz}

\allowdisplaybreaks

\newtheorem{thm}{Theorem}[section]
\newtheorem{prop}[thm]{Proposition}
\newtheorem{conj}[thm]{Conjecture}
\newtheorem{cor}[thm]{Corollary}
\newtheorem{lem}[thm]{Lemma}

\theoremstyle{definition}

\numberwithin{equation}{section}
\newtheorem{rem}[thm]{\bf Remark}
\newtheorem{ex}[thm]{\bf Example}

\newtheorem{defn}[thm]{\bf Definition}

\def\bbQ{\mathbb{Q}}
\def\bbR{\mathbb{R}}

\def\bbZ{\mathbb{Z}}

\def\bfPsi{\boldsymbol{\Psi}}

\def\bfe{\mathbf{e}}

\def\bfn{\mathbf{n}}

\def\bfw{\mathbf{w}}

\def\bfz{\mathbf{z}}

\def\rmid{\mathrm{id}}

\def\rmpr{\mathrm{pr}}

\def\rmSupp{\mathrm{Supp}}

\def\calR{\mathcal{R}}
\def\calS{\mathcal{S}}

\def\frakd{\mathfrak{d}}
\def\frakD{\mathfrak{D}}

\def\frakg{\mathfrak{g}}

\def\frakj{\mathfrak{j}}
\def\frakp{\mathfrak{p}}

\def\fraks{\mathfrak{s}}

\def\l{\ell}
\def\ss{\scriptstyle}

\def\d{\delta}

\begin{document}
\bibliographystyle{amsalpha}

\title[Pentagon relation in QCSD]{
Pentagon relation in
quantum cluster scattering diagrams}
\address{\noindent Graduate School of Mathematics, Nagoya University, 
Chikusa-ku, Nagoya,
464-8604, Japan}
\email{nakanisi@math.nagoya-u.ac.jp}
\author{Tomoki Nakanishi}
\maketitle

\begin{abstract}
We formulate the pentagon relation for quantum dilogarithm elements in the structure group of a quantum cluster scattering diagram (QCSD). As an application, we establish
 the nonpositivity of a certain class of nonskew-symmetric QCSDs. Also, we explicitly present various consistency relations and their positivity for  QCSDs of rank 2 completely or up to some degree, most of which are new in the literature.
\end{abstract}

\section{Introduction}

In  \cite{Gross14} 
the scattering diagram method 
 was employed
to study cluster algebras \cite{Fomin02},
where all information of a cluster algebra or a cluster pattern is encoded in the corresponding \emph{cluster
scattering diagram} (CSD, for short).
It was clarified implicitly in \cite{Gross14} and more explicitly in \cite{Nakanishi22a} that
the \emph{dilogarithm elements} $\Psi[n]$,
where $n$ is its normal vector,
and the \emph{pentagon relation} among them play prominent roles
in a CSD.
Meanwhile,
it has been  known that
cluster algebras 
admit natural quantizations \cite{Berenstein05b, Tran09, Fock03, Fock07},
where the \emph{quantum dilogarithm} plays a key role.
Naturally, \emph{quantum cluster scattering diagrams} (QCSD, for short),
which are  quantum analogs of  CSDs,
were introduced and  studied in \cite{Mandel15, Davison19},
where \emph{quantum dilogarithm elements} 
play a primary role.
Here and below, we distinguish the quantum dilogarithm $\bfPsi_q(x)$ and the quantum dilogarithm elements $\Psi_{a,b}[n]$.
The former is a formal power series of an indeterminate $x$ with a parameter $q$ due to \cite{Schutzenberger53,Faddeev94,Fock03}, while the latter
are certain elements of the structure group $G$ of a QCSD.
It is well known that the quantum dilogarithm satisfies the \emph{pentagon identity} \cite{Faddeev94,Chekhov99},
which is a quantum analog of the pentagon identity (more often called Abel's identity or the five-term relation) of the dilogarithm \cite{Lewin81}.

Having this development in mind,
in this paper we  formulate the \emph{pentagon relation} among 
the quantum dilogarithm elements (Theorem \ref{thm:pent1}),
which is the counterpart of the pentagon relation for the dilogarithm elements for a CSD in \cite{Nakanishi22a}.
The definition of QCSDs is in parallel with the one for CSDs
thanks to a more general formulation of scattering diagrams by \cite{Kontsevich13}.
However, each quantum dilogarithm element $\Psi_{a,b}[n]$
 carries some additional rational parameters (the \emph{quantum data}) $a$ and $b$
compared  with its classical counterpart $\Psi[n]$,
and this is the source of the complication and the richness
in the quantum case.
(The situation is  somewhat similar  to  the \emph{fusion} of representations of  quantum affine algebras
with the \emph{spectral parameter}.)

As an application, we study the positivity of QCSDs in Section 5.
We first formulate the positivity of CSDs and QCSDs (Definitions \ref{defn:positiveCSD1} and \ref{defn:positive1}) based on the structure of the walls therein.
It is known that any CSD is positive \cite{Gross14} and it implies the positivity of the theta functions for each CSD.
It is also known that any QCSD with the \emph{skew-symmetric} initial exchange matrix is positive \cite{Davison19}.
On the other hand, there are examples of nonpositive theta functions for QCSDs \cite{Lee12, Cheung20}.
This implies that QCSDs are not always positive in the \emph{nonskew-symmetric} case.
Using the pentagon relation, we clarify the origin of the nonpositivity in the combinatorics of quantum dilogarithm elements,
and we establish the \emph{nonpositivity} of a certain class of nonskew-symmetric QCSDs
  (Theorem \ref{thm:nonpos1} and 
Corollary
 \ref{cor:nonpos2}).
 This extends the results of \cite{Lee14, Davison19, Cheung20} systematically.
Also, we explicitly present various consistency relations
and their positivity for  QCSDs of rank 2
completely (Section \ref{subsec:rank2}) or up to degree 4 (Section \ref{subsec:positivity2}).
Most of these formulas are new in the literature and exhibit the complexity of the positivity phenomenon.

\medskip
\noindent
{Acknowledgment.}
We thank Peigen Cao for useful comments and especially for pointing out the errors in Section \ref{subsec:reduction1} in the early manuscript.
This work is partly supported by JSPS Grant Numbers JP16H03922 and JP22H01114.

\section{Quantum dilogarithm elements and pentagon relation}
\label{sec:qunatum1}

We introduce quantum dilogarithm elements and the pentagon relation among them.

\subsection{Structure group $G$}
\label{subsec:structure1}
Here we introduce the underlying \emph{structure group} $G$ of a  QCSD
 following \cite{Kontsevich13,Mandel15,Davison19, Cheung20}. 
See also   \cite{Gross14, Nakanishi22a} for parallel notions for a CSD.

Let $\Gamma$ be a \emph{fixed data} consisting of the following:
\begin{itemize}
\item a lattice $N$ of rank $r$,
\item a skew-symmetric bilinear form $\{\cdot, \cdot\}\colon N \times N \rightarrow \bbQ$,
\item a sublattice $N^{\circ}\subset N$ of rank $r$ such that $\{N^{\circ}, N\}\subset \bbZ$,
\item positive integers $\d_1$, \dots, $\d_r$ such that
there is a basis $(e_1,\dots,e_r)$ of $N$, where $(\d_1e_1,\dots,\d_re_r)$
is a basis of $N^{\circ}$,
\item $M=\mathrm{Hom}(N, \bbZ)$ and $M^{\circ}=\mathrm{Hom}(N^{\circ}, \bbZ)$.
\end{itemize}
Let $M_{\bbR}=M\otimes_{\bbZ} \bbR$.
Let
$\langle n, m\rangle$ denote  the canonical paring either for $N^{\circ}\times M^{\circ}$
or for $N\times M_{\bbR}$.
For $n\in N$, $n\neq 0$, let $n^{\perp}:=\{ \xi\in M_{\bbR} \mid \langle n, \xi\rangle =0\}$.

Let $\fraks=(e_1,\dots, e_r)$ be a \emph{seed} for $\Gamma$,
which is a basis of $N$ such that $(\d_1e_1,\dots, \d_re_r)$ is a basis of $N^{\circ}$.
The dual bases of $M$ and $M^{\circ}$ are given by
$(e_1^*,\dots,e_r^*)$ and 
$(f_1,\dots, f_r):=(e_1^*/\d_1,\dots,e_r^*/\d_r)$, respectively.
The initial exchange matrix  $B=(b_{ij})$ of the corresponding quantum cluster algebra
is given by
\begin{align}
\label{eq:B1}
b_{ij}=\{\d_i e_i, e_j\}.
\end{align}

Let 
\begin{align}
N^+=\Biggl\{ \sum_{i=1}^r a_i e_i \mid a_i \in \bbZ_{\geq 0},\, \sum_{i=1}^r a_i >0\Biggr\}
\end{align}
be the positive vectors of $N$ with respect to $\fraks$.
Let $N^+_{\rmpr}$ denote the set of primitive elements in $N^+$.
The degree function $\deg\colon N^+\rightarrow \bbZ_{> 0}$ is defined by $\deg( \sum_{i=1}^r a_i e_i ):=\sum_{i=1}^r a_i$.

For $n\in N^+$,  let $\d(n)$ be the smallest positive rational number 
such that $\d(n) n \in N^{\circ}$,
which is called the \emph{normalization factor} of $n$.
For example, $\d(e_i)=\d_i$.
We have $\d(n)\in \bbZ_{>0}$ and $\d(tn)=\d(n)/t$ for any $n\in N^+_{\rmpr}$ and $t\in \bbZ_{>0}$.
We set $\delta_0$ as the least common multiple of $\d_1$, \dots, $\d_r$.
Then, $\delta_0n\in N^{\circ}$ for any $n\in N^+$. It implies that
$\{\cdot, \cdot\}\in (1/\delta_0)\bbZ$
and that $\delta_0/\d(n)\in \bbZ_{>0}$.

Let $q$ be an indeterminate, and let $\bbQ(q^{ 1/\delta_0})$ be the rational function field of $q^{1/\delta_0}$.
For any $\alpha\in (1/\delta_0)\bbZ$, let
\begin{align}
[\alpha]_q:=\frac{q^{\alpha}-q^{-\alpha}}{q-q^{-1}}\in \bbQ(q^{ 1/\delta_0})
\end{align}
be the \emph{$q$-number}, which has the limit 
\begin{align}
\label{eq:qnumber1}
\lim_{q\rightarrow1}[\alpha]_q=\alpha.
\end{align}
Let $\frakg$ be the $N^+$-graded Lie algebra over $\bbQ(q^{ 1/\delta_0})$ defined by
\begin{align}
\label{eq:Xcom1}
\frakg=\bigoplus_{n\in N^+} \bbQ(q^{ 1/\delta_0}) X_n,
\quad
[X_n, X_{n'}]:=[\{n,n'\}]_q X_{n+n'},
\end{align}
where the Jacobi identity is easily verified.
Let $\widehat \frakg$ be the completion of $\frakg$ with respect to $\deg$, and let $G=\exp(\widehat \frakg)$
be the exponential group of $\widehat \frakg$ whose product is defined by the Baker-Campbell-Hausdorff formula
(e.g., \cite[\S V.5]{Jacobson79}).
We call $G$ the \emph{structure group} for the forthcoming scattering diagrams.

Let $(N^+)^{>\ell}=\{n\in N^+ \mid \deg(n)>\ell\}$.
Let $G^{>\ell}$ be the normal subgroup of $G$
consisting elements 
$\exp(\sum_{n\in (N^+)^{>\ell}}
c_n X_n)$, where the sum may be infinite,
and let $G^{\leq \ell}:=G/G^{>\ell}$ be its quotient.
For each $n\in N_{\rmpr}^+$, let $G_n^{\parallel}$ be the abelian subgroup of $G$ 
consisting of elements $\exp(\sum_{j=1}^{\infty}c_{j}X_{jn})$.
For any $g=\exp(X)$ and $c\in \bbQ(q^{ 1/\delta_0})$, 
the power $g^c\in G$ is defined by $\exp(cX)$.

\begin{rem}
So far, 
the only  difference from the classical case is the Lie bracket
in \eqref{eq:Xcom1}.
It is a quantum analog of the Lie bracket 
\begin{align}
\label{eq:cX1}
 [X_n, X_{n'}]:=\{n,n'\} X_{n+n'}.
 \end{align}
for a CSD.
In what follows all expressions involving $q$ converge to their classical counterparts
in the limit $q\rightarrow 1$.
\end{rem}
\subsection{$y$-representation}
\label{subsec:yrep1}
Let $y$ be a symbol.
Let $\bbQ(q^{ 1/\delta_0})[y]_q$ be the noncommutative and associative algebra over
$\bbQ(q^{ 1/\delta_0})$ with generators $y^n$ ($ n\in N^+\sqcup \{0\})$
and the $q$-commutative elations
\begin{align}
\label{eq:yrel1}
 y^n y^{n'}=q^{\{n,n'\}} y^{n+n'}.
\end{align}
Let $\calR_q(y)$ be the completion of $\bbQ(q^{ 1/\delta_0})[y]_q$ with respect to  $\deg$.
In other words,
any element of $\calR_q(y)$ is expressed as an infinite sum
\begin{align}
\sum_{n\in N^+\sqcup \{0\}}
c_n y^n
\quad
(c_n \in \bbQ(q^{ 1/\delta_0})).
\end{align}
Then, we define the action of  $\widehat\frakg$ on $\calR_q(y)$ by
\begin{align}
\label{3q:Xact1}
X_n(y^{n'}):= &\ [\{n,n'\}]_q y^{n'+n}\\
\label{3q:Xact12}
=&\  \frac{q^{2\{n,n'\}}-1}{q-q^{-1}}y^{n'}y^{n}.
\end{align}
It is easy to check that this is indeed an action of  $\widehat\frakg$
and also a derivation.
Thus,
it induces the action  of $G$ on  $\calR_q(y)$ defined by
\begin{align}
\label{eq:Xaction1}
(\exp X)(y^{n})=
\sum_{j=0}^{\infty}
\frac{1}{j!} X^j(y^{n})
\quad
(X\in \widehat\frakg).
\end{align}
Moreover, $\exp X$ is an algebra automorphism
\cite[\S I.2]{Jacobson79}.
We call the resulting representation
 $\rho_y: G\rightarrow \mathrm{Aut}(\calR_q(y))$
 the \emph{(quantum) $y$-representation} of $G$.
It is faithful if and only if $\{\cdot, \cdot\}$ is
nondegenerate.

\begin{rem}
If we replace the relation \eqref{eq:yrel1} with
the standard one
$ y^n y^{n'}=y^{n+n'}$,
the action \eqref{3q:Xact1} is still an action of $\widehat{\frakg}$,
but it is not a derivation.
This explains the necessity of the relation  \eqref{eq:yrel1}.
\end{rem}

\subsection{Quantum dilogarithm elements and pentagon relation}
\label{subsec:quantum1}

Let $\bfPsi_q(x)$ be the \emph{quantum dilogarithm} by \cite{Schutzenberger53, Faddeev93, Faddeev94, Fock03},
\begin{align}
\label{eq:qd1}
\bfPsi_q(x):=&\
\exp
\biggr(
\sum_{j=1}^{\infty}
\frac{(-1)^{j+1}}
{j(q^j - q^{-j})}
x^j
\biggl)
\quad
\in \bbQ(q)[[x]]
\\
=&\
\prod_{j=0}^{\infty}
(1+q^{2j+1} x)^{-1}.
\end{align}
This convention is due to \cite{Fock03}.
The equality holds, because
both expressions satisfy the  property
\begin{align}
\bfPsi_q(q^2 x)=(1+qx)\bfPsi_q(x),
\quad
\bfPsi_q(0)=1,
\end{align}
which uniquely determines $\bfPsi_q(x)$.
See, for example, \cite{Kirillov95, Fock03} for further information.

The following definition  is similar to   the first expression \eqref{eq:qd1}
of $\bfPsi_q(x)$.

\begin{defn}[Quantum dilogarithm element]
For any $n\in N^+$, $a\in (1/\delta_0)\bbZ_{>0}$, and  $b\in (1/\delta_0)\bbZ$, we define
\begin{align}
\label{eq:diloge1}
\Psi_{a,b}[n]:=\exp\Biggl( \sum_{j=1}^{\infty} \frac{(-1)^{j+1}}{j [ja]_q} q^{jb}X_{jn}\Biggr) \in G_{n_0}^{\parallel},
\end{align}
where $n_0\in  N_{\rmpr}^+$ is the one such that $n=jn_0$ for some $j\in \bbZ_{>0}$.
We call it a  \emph{quantum dilogarithm element}.
We call the parameters $a$ and $b$ the \emph{quantum data} of $\Psi_{a,b}[n]$.
They control the \emph{interval} and the \emph{shift}  in the sum, respectively.
For simplicity, we also write
\begin{align}
\Psi_{a}[n]:=\Psi_{a,0}[n].
\end{align}
\end{defn}

\begin{rem}
The above  $\Psi_{a,b}[n]$ is a quantum analog of the dilogarithm element 
in 
the classical case \cite[\S III.1.4]{Nakanishi22a}
\begin{align}
\Psi[n]
=\exp \Biggl(
\sum_{j=1}^{\infty}  \frac{ (-1)^{j+1} }{j^2}X_{jn}\Biggr).
\end{align}
We have
\begin{align}
\lim_{q\rightarrow 1} \Psi_{a,b}[n] = \Psi[n]^{1/a}.
\end{align}
Note that the interval $a$ appears in the RHS as the inverse.
\end{rem}

\begin{prop}
\label{prop:generate1}
The group $G$ is generated by $\Psi_{a,b}[n]^c$ {\upshape ($n\in N^+$,
$a\in (1/\delta_0)\bbZ_{>0}$, $b\in (1/\delta_0)\bbZ$, $c\in \bbQ(q^{ 1/\delta_0})$)}
admitting the infinite product.
\end{prop}
\begin{proof}
By \eqref{eq:diloge1}, any element of
$G_{n_0}^{\parallel}$ is expressed as an infinite product of 
$\Psi_{a,b}[jn_0]^c$'s.
\end{proof}

By \eqref{3q:Xact12},
each dilogarithm element  $\Psi_{a,b}[n]$ acts on  $\calR_q(y)$ under $\rho_y$ as
\begin{align}
\label{eq:ymut1}
\begin{split}
\Psi_{a,b}[n](y^{n'})
&= \exp\Biggl(\,\sum_{j=1}^{\infty}
 \frac{(-1)^{j+1}}{j[ja]_q} q^{jb} X_{jn}\Biggr)(y^{n'})\\
&=y^{n'}  \exp\Biggl( \sum_{j=1}^{\infty} 
 \frac{q^{2j\{ n,n'\}}-1}{q^{2ja}-1} \frac{(-1)^{j+1}}{j }q^{ja}q^{jb} y^{ jn}\Biggr).
 \end{split}
 \end{align}
 We consider two special situations,
 where the expression \eqref{eq:ymut1} is simplified.
 
\par
(i). Fix $n$ in \eqref{eq:ymut1}. Let $a=1/s\d(n)$, where $s\in \bbZ_{>0}$ is a divisor of $\delta_0/\d(n)$.
Then, $a\in (1/\d_0)\bbZ_{>0}$,
and also $\alpha:=\{s\d(n)n,n'\}$ is  an integer.
Then, we have
 \begin{align}
 \label{eq:qq0}
 \begin{split}
  \frac{q^{2j\{ n,n'\}}-1}{q^{2ja}-1} 
  &=
  \frac{q^{2\alpha j/s\d(n)}-1}{q^{2j/s\d(n)}-1} 
 =
  \begin{cases}
\displaystyle
\sum_{p=0}^{\alpha-1}
q^{2jp/s\d(n)}
& \alpha> 0,
\\
0
&
\alpha=0,
\\
\displaystyle
-
\sum_{p=1}^{-\alpha}
q^{-2jp/s\d(n)}
& \alpha< 0.
\end{cases}
\end{split}
 \end{align}
 Thus, we obtain
 \begin{align}
 \label{eq:ymut2}
 \begin{split}
 &\quad \
 \Psi_{1/s\d(n),b}[n](y^{n'})   
=
\begin{cases}
\displaystyle
y^{n'} \prod_{p=1}^{\alpha}
(1+q^{(2p-1)/s\d(n)}q^b y^n)
& \alpha> 0,
\\
y^{n'}
& \alpha=0,
\\
\displaystyle
y^{n'} \prod_{p=1}^{-\alpha}
(1+q^{-(2p-1)/s\d(n)}q^b y^n)^{-1}
& \alpha< 0.
\end{cases}
\end{split}
\end{align}
Observe that
this is the automorphism part of the Fock-Goncharov decomposition
of mutations of the quantum $y$-variables in \cite[\S 3.3]{Fock03}.

 (ii). Fix $n$ and $n'$ in \eqref{eq:ymut1}.
 Suppose that $\{n,n'\}\neq 0$,
 and let $a=|\{n,n'\}|$.
 Then,
 we have
 \begin{align}
 \begin{split}
  \frac{q^{2j\{ n,n'\}}-1}{q^{2ja}-1} 
  &=
  \begin{cases}
1
& \{n,n'\}>0,
\\
\displaystyle
-
q^{-2ja}
& \{n,n'\}<0.
\end{cases}
\end{split}
 \end{align}
 Thus, we obtain
 \begin{align}
 \label{eq:ymut3}
 \begin{split}
 &\quad \
 \Psi_{a,b}[n](y^{n'})   
=
\begin{cases}
\displaystyle
y^{n'} 
(1+q^{a}q^b y^n)
&\{n,n'\}>0,
\\
\displaystyle
y^{n'} 
(1+q^{-a}q^b y^n)^{-1}
& \{n,n'\}<0.
\end{cases}
\end{split}
\end{align}

\begin{rem}
\label{rem:adjoint1}
 The same formula in the RHS of  \eqref{eq:ymut2} 
 is also obtained by the
 \emph{adjoint action} of the
 quantum dilogarithm
 $\bfPsi_{q^{1/s\delta(n)}}(q^by^n)$ on $\calR_q(y)$
 \cite{Fock03}.
 In this sense, we have the correspondence between
  $\Psi_{1/s\d(n),b}[n]$ and
 the quantum dilogarithm
  $\bfPsi_{q^{1/s\delta(n)}}(q^by^n)$.
\end{rem}

We have the following important application of the formula
 \eqref{eq:ymut3}.
 
\begin{thm}
\label{thm:pent1}
Let $n_1,n_2\in N^+$.
The following relations hold in $G$.
\par
(a). Suppose that $\{n_2,n_1\}=0$.
Then,
for any $a_1$, $a_2$, $b_1$, and $b_2$,
we have
\begin{align}
\label{eq:com1}
\Psi_{a_2,b_2}[ n_2] \Psi_{a_1,b_1}[n_1 ] 
=\Psi_{a_1,b_1}[n_1 ] \Psi_{a_2,b_2}[ n_2] .
\end{align}

(b). (Pentagon relation. Cf.\ \cite[\S 2]{Faddeev94}).
Suppose that $\{n_2,n_1\}=c$ $(c\in (1/\delta_0)\bbZ_{>0})$.
Then,
for any $b_1$ and $b_2$,
we have
\begin{align}
\label{eq:pent1}
&\Psi_{c,b_2} [n_2 ] \Psi_{c,b_1}[ n_1]
=
\Psi_{c,b_1}[  n_1 ] \Psi_{c,b_1+b_2}[n_1+n_2] \Psi_{c,b_2}[n_2].
\end{align}
\end{thm}
\begin{proof}
The relation \eqref{eq:com1} is clear by \eqref{eq:Xcom1}.
Let us prove the relation \eqref{eq:pent1}.
Since it only involves
$\Psi_{a,b}[n]$'s for the rank 2 sublattice $N'$ of $N$
generated by $n_1$ and $n_2$,
we concentrate on the subgroup $G'$ of $G$ corresponding to $N'$.
Accordingly, we  consider 
the $y$-representation of $G'$ acting on the subalgebra  of  $\calR_q(y)$ 
generated by $y^{n_1}$ and $y^{n_2}$.
By the assumption $\{n_2,n_1\}\neq 0$, this  is faithful.
Thus, one can prove  \eqref{eq:pent1} with this representation.

First, we consider the action on $y^{n_1}$.
Note that $\{n_1+n_2,n_1\}=-\{n_1+n_2,n_2 \}=c>0$.
Then,  thanks to the formula \eqref{eq:ymut3},
we have
\begin{align*}
\begin{split}
(\Psi_{c,b_2}[n_2 ] \Psi_{c,b_1}[ n_1])(y^{n_1})
&=
\Psi_{c,b_2}[n_2 ](y^{n_1})
\\
&=
y^{n_1} (1+ q^{c+b_2} y^{n_2}),
\end{split}
\end{align*}
and also
\begin{align*}
\begin{split}
&\quad\
(\Psi_{c,b_1}[n_1 ] \Psi_{c,b_1+b_2}[  n_1+n_2]
 \Psi_{c,b_2}[n_2])(y^{n_1})
 \\
&=
(\Psi_{c,b_1}[n_1 ] \Psi_{c,b_1+b_2}[  n_1+n_2])
(y^{n_1}(1+q^{c+b_2} y^{n_2}))
\\
&=
\Psi_{c,b_1}[n_1 ] 
(y^{n_1} (1+ q^{c+b_1+b_2} y^{n_1+n_2})
\\
& \qquad \times (1+q^{c+b_2} y^{n_2}(1+q^{-c+b_1+b_2} y^{n_1+n_2})^{-1}))
\\
&=
\Psi_{c,b_1}[n_1 ] 
(y^{n_1} (1+q^{c+b_2} y^{n_2}+q^{c+b_1+b_2} y^{n_1+n_2}))
\\
&=
y^{n_1} (1+q^{c+b_2} y^{n_2}(1+q^{-c+b_1}y_1)^{-1}+q^{c+b_1+b_2} y^{n_1+n_2}(1+q^{-c+b_1}y_1)^{-1})
\\
&=
y^{n_1} (1+q^{-c+b_1}y_1+q^{c+b_2} y^{n_2}+q^{b_1+b_2} y^{n_2}y^{n_1})(1+q^{-c+b_1}y_1)^{-1}
\\
&=
y^{n_1} (1+ q^{c+b_2} y^{n_2}).
\end{split}
\end{align*}
Thus, they coincide. 
Next, we consider the action on $y^{n_2}$.
Similarly, we have
\begin{align*}
\begin{split}
&\quad \
(\Psi_{c,b_2}[n_2 ] \Psi_{c,b_1}[ n_1])(y^{n_2})
\\
&=
\Psi_{c,b_2}[n_2 ](y^{n_2}(1+q^{-c+b_1}y^{n_1})^{-1})
\\
&=
y^{n_2} (1+q^{-c+b_1}y^{n_1}(1+ q^{c+b_2} y^{n_2}))^{-1}
\\
&=
y^{n_2} (1+q^{-c+b_1}y^{n_1}+ q^{b_1+b_2} y^{n_1}y^{n_2})^{-1},
\end{split}
\end{align*}
and also
\begin{align*}
\begin{split}
&\quad\
(\Psi_{c,b_1}[n_1 ] \Psi_{c,b_1+b_2}[  n_1+n_2]
 \Psi_{c,b_2}[n_2])(y^{n_2})
 \\
&=
(\Psi_{c,b_1}[n_1 ] \Psi_{c,b_1+b_2}[  n_1+n_2])
(y^{n_2})
\\
&=
\Psi_{c,b_1}[n_1 ]
(y^{n_2} (1+ q^{-c+b_1+b_2} y^{n_1+n_2})^{-1})
\\
&=
y^{n_2} (1+q^{-c+b_1}y^{n_1})^{-1}(1+ q^{-c+b_1+b_2} y^{n_1+n_2}(1+q^{-c+b_1}y_1)^{-1})^{-1}
\\
&=
y^{n_2} (1+q^{-c+b_1}y^{n_1}+ q^{b_1+b_2} y^{n_1}y^{n_2})^{-1}.
\end{split}
\end{align*}
Again, they coincide.
\end{proof}

Also, the following operations to decompose and unify quantum dilogarithm elements
are important for our purpose.
\begin{prop}
Let $p$ be a positive integer.
The following equalities hold,
where we assume  $a/p \in (1/\d_0)\bbZ_{>0}$ in the second equality:
\begin{align}
\label{eq:fis1}
&\text{\rm (fission)}
\quad
\Psi_{a,b}[n]
=
\prod_{t=1}^p \Psi_{pa, b+(2t-p-1)a}[n],
\\
\label{eq:fus1}
&\text{\rm (fusion)}
\quad
\prod_{t=1}^p \Psi_{a, b+(2t-p-1)a/p}[n]
=
\Psi_{a/p,b}[n].
\end{align}
\end{prop}
\begin{proof}
Two equalities \eqref{eq:fis1} and \eqref{eq:fus1}
 are equivalent.
 The equality \eqref{eq:fis1} is shown as follows:
\begin{align*}
\begin{split}
\Psi_{a,b}[n]
&=
\exp\Biggl( \sum_{j=1}^{\infty} \frac{(-1)^{j+1}}{j [ja]_q} 
\frac{q^{jpa}-q^{-jpa}}{q^{jpa}-q^{-jpa}} 
q^{bj}X_{jn}\Biggr) 
\\
&=
\exp\Biggl( \sum_{j=1}^{\infty} \frac{(-1)^{j+1}}{j [jpa]_q} 
\frac{q^{jpa}-q^{-jpa}}{q^{ja}-q^{-ja}} 
q^{bj}X_{jn}\Biggr) 
\\
&=
\exp\Biggl( 
\sum_{t=1}^p
\sum_{j=1}^{\infty} \frac{(-1)^{j+1}}{j [jpa]_q} 
q^{j(2t-p-1)a}
q^{bj}X_{jn}\Biggr) 
\\
&=\prod_{t=1}^{p}
\Psi_{pa, b+(2t-p-1)a}[n].
\end{split}
\end{align*}
\end{proof}

\section{Quantum cluster scattering diagrams}
\label{sec:ex1}

We introduce quantum cluster scattering diagrams.
Examples in rank 2 are also given.

\subsection{Quantum cluster scattering diagrams}
\label{subsec:quantum2}

The  following definitions for (quantum) scattering diagrams
are just the same as the classical one in \cite{Gross14},
by replacing the structure group $G$ with the one in
Section \ref{subsec:structure1}.

A \emph{wall} $\bfw=(\frakd, g)_{n}$ for $\fraks$ is a triplet with
$n\in N_{\rmpr}^+$, a cone $\frakd \subset n^{\perp}\subset M_{\bbR}$ of codimension 1, and $g\in G_n^{\parallel}$. 
We call $n$, $\frakd$, $g$, the \emph{normal vector}, the \emph{support}, the \emph{wall element}
of $\bfw$, respectively.
Let $p^*\colon N \rightarrow M^{\circ}\subset M_{\bbR}$, $n \mapsto \{ \cdot, n\}$.
We say that a wall $\bfw=(\frakd, g)_{n}$ is \emph{incoming} if $p^*(n)\in \frakd$ holds. 

\begin{defn}[Scattering diagram]
A \emph{scattering diagram} $\frakD=\{ \bfw_{\lambda}=(\frakd_{\lambda}, g_{{\lambda}})_{n_{\lambda}}
\}_{\lambda\in \Lambda}$ for $\fraks$ is a collection of walls for $\fraks$
satisfying the  following finiteness condition:
For any degree $\ell$,
there are only finitely many walls such that $\pi_{\ell}(g_{\lambda})\neq \rmid$,
where $\pi_{\ell}\colon G \rightarrow G^{\leq \ell}$ is the canonical projection.
\end{defn}

For a scattering diagram  $\frakD$, we define
\begin{align}
\mathrm{Supp}(\frakD)&=\bigcup_{\lambda\in \Lambda} \frakd_{\lambda},
\quad
\mathrm{Sing}(\frakD)=\bigcup_{\lambda\in \Lambda} \partial\frakd_{\lambda}
\cup
\bigcup_{\ss \lambda, \lambda'\in \Lambda \atop \ss \dim 
\frakd_{\lambda}\cap \frakd_{\lambda'}=r-2
} \frakd_{\lambda}\cap \frakd_{\lambda'}.
\end{align}

A curve $\gamma\colon [0,1]\rightarrow M_{\bbR}$ is
\emph{admissible} for  $\frakD$
if it satisfies the following properties:
\begin{itemize}
\item[(1)]
The endpoints of $\gamma$ are in $M_{\bbR}\setminus \mathrm{Supp}(\frakD)$.
\item[(2)]
It is a smooth curve, and it intersects $\mathrm{Supp}(\frakD)$ transversally.
\item[(3)]
$\gamma$ does not intersect  $\mathrm{Sing}(\frakD)$.
\end{itemize}
For any admissible curve $\gamma$,
the path-ordered product $\frakp_{\gamma, \frakD}\in G$ is defined
as the product of
the wall elements $g_{{\lambda}}^{\epsilon_{\lambda}}$ of walls $\bfw_{\lambda}$ of $\frakD$ intersected by $\gamma$
in the order of intersection, where $\epsilon_{\lambda}$ is the \emph{intersection sign}
defined by
\begin{align}
\label{eq:int1}
\epsilon_{\lambda}=
\begin{cases}
1 & \langle n_{\lambda}, \gamma'\rangle <0,\\
-1 & \langle n_{\lambda}, \gamma'\rangle >0,
\end{cases}
\end{align}
and $\gamma'$ is the velocity vector of $\gamma$ at the wall $\bfw_{\lambda}$.
The product $\frakp_{\gamma, \frakD}$ is infinite in general,
and it is well-defined in $G$ due to the finiteness condition.
We say that a pair of scattering diagrams $\frakD$  and $\frakD'$ are \emph{equivalent} if  $\frakp_{\gamma, \frakD}=\frakp_{\gamma, \frakD'}$
for any admissible curve $\gamma$  for both $\frakD$ and $\frakD'$.
We say that a scattering diagram $\frakD$ is \emph{consistent} if  $\frakp_{\gamma, \frakD}=\rmid$
for any admissible loop (i.e., closed curve) $\gamma$ for $\frakD$.

The following definition is due to \cite{Mandel15}.

\begin{defn}[Quantum cluster scattering diagram]
A \emph{quantum cluster scattering diagram} $\frakD^q_{\fraks}$ (QCSD, for short) for $\fraks$ is a consistent scattering diagram 
whose set of incoming walls are given by
\begin{align}
\mathrm{In}_{\fraks}:=\{ (e_i^{\perp}, \Psi_{1/\d_i}[e_i])_{e_i} \mid i=1,\dots, r\}.
\end{align}
\end{defn}

\begin{thm}[{\cite[Theorems 1.12]{Gross14}, \cite[Theorem 2.1.6]{Kontsevich13}, \cite[Theorem 2.13]{Davison19}}]
\label{thm:CSD1}
There exists a QCSD $\frakD^q_{\fraks}$ uniquely up to equivalence.
\end{thm}
\begin{proof}
The proof of \cite[Theorem 1.12]{Gross14}, which originates in \cite[\S2]{Kontsevich13}, is also applicable to 
this case.
See also \cite[\S III.3]{Nakanishi22a}, which is closer to the present setting.
\end{proof}

Thanks to Proposition \ref{prop:generate1},
we may assume that, up to equivalence, wall elements of $\frakD_{\fraks}^q$
are given by  $\bbQ(q^{1/\delta_0})$-powers of quantum dilogarithm elements $\Psi_{a,b}[n]^c$
($n\in N^+$, $c\in \bbQ(q^{1/\delta_0})$).

\begin{rem}
To each admissible loop $\gamma$,
the consistency condition $\frakp_{\gamma, \frakD}=\rmid$
gives a relation among quantum dilogarithm elements.
Under the correspondence in Remark \ref{rem:adjoint1}
this relation is  translated as an identity for the quantum dilogarithm
(\emph{quantum dilogarithm identity}),
which is a generalization of a quantum dilogarithm identity associated 
with a period of a quantum cluster pattern studied in 
\cite{Fock07, Keller11, Kashaev11, Nagao11b}.
It is also the quantum counterpart of the (classical) dilogarithm identity
for the Rogers dilogarithm
 in \cite{Nakanishi21d}.
\end{rem}

\subsection{Rank 2 examples}
\label{subsec:rank2}
Let us demonstrate to construct the rank 2 QCSDs of finite and affine types
based on the pentagon relation \eqref{eq:pent1}
in the  spirit of \cite[\S III.2.2, \S III.3.5]{Nakanishi22a}
 for their classical counterparts.
 
\subsubsection{Finite type}
Here we consider the case where $\{\cdot, \cdot\}\neq 0$.
Without loss of generality, we may assume that $\{e_2, e_1\}=1$
and $\d_2 \geq \d_1$.
 Thus,
the initial exchange matrix $B$ in \eqref{eq:B1} is given by
\begin{align}
B=
\begin{pmatrix}
0 & -\d_1\\
\d_2 & 0
\end{pmatrix}.
\end{align}

Let $e_1^*, e_2^*\in M$ be the dual basis of $e_1, e_2\in N$.
Accordingly,  let $f_1= e_1^*/\d_1, f_2= e_2^*/\d_2 \in M^{\circ}$ be the dual basis of $
\d_1 e_1, \d_2 e_2\in N^{\circ}$.
Let $\bfe_1, \bfe_2 \in \bbZ^2$ be the unit vectors.
We identify $N \simeq \bbZ^2$, $e_i \mapsto \bfe_i$,
and $M_{\bbR} \simeq \bbR^2$, $f_i \mapsto \bfe_i$.
Then, we have $\{\bfn',\bfn \}=n'_2 n_1 - n'_1 n_2$.
Also,
the canonical paring
$\langle n, z \rangle\colon
 N\times M_{\bbR}\rightarrow \bbR$
   is given 
   by the corresponding vectors $\bfn$ and $\bfz$ as
\begin{align}
\label{3eq:can1}
\langle \bfn, \bfz \rangle
=
\bfn^T
\begin{pmatrix}
\d_1^{-1} & 0 \\
0 & \d_2^{-1} 
\end{pmatrix}
\bfz.
\end{align}
For $\bfn=(n_1,n_2)$, we write
\begin{align}
\begin{bmatrix}
n_1\\
n_2
\end{bmatrix}
_{a,b}
:=
\Psi_{a,b}[ \bfn],
\quad
\begin{bmatrix}
n_1\\
n_2
\end{bmatrix}
_{a}
:=
\Psi_{a}[ \bfn].
\end{align}

\begin{figure}
\centering
\leavevmode
\begin{xy}
0;/r1.2mm/:,
(0,-24)*{\text{(a) $A_2$}},
(4,9)*{\gamma_1};
(9,4)*{\gamma_2};
(7,7)*+{\bullet};(-7,-7)*+{\bullet};
(7,7)*+{};(-7,-7)*+{}
   **\crv{(0,6.7)&(-6.7,0)}
          ?>*\dir{>};
       (7,7)*+{};(-7,-7)*+{}
       **\crv{(6.7,0)&(0,-6.7)}
       ?>*\dir{>};
(-4, 9)*{\text{\small $\begin{bmatrix}1\\0\end{bmatrix}_{1}$}},
(-13, 0)*{\text{\small $\begin{bmatrix}0\\1\end{bmatrix}_{1}$}},
(11, -16)*{\text{\small $\begin{bmatrix}1\\1\end{bmatrix}_{1}$}},
(0,0)="A",
\ar@{-} "A"+(0,0); "A"+(10,0)
\ar@{-} "A"+(0,0); "A"+(0,10)
\ar@{-} "A"+(0,0); "A"+(-10,0)
\ar@{-} "A"+(0,0); "A"+(0,-10)
\ar@{-} "A"+(0,0); "A"+(10,-10)
\end{xy}
\hskip12pt
\begin{xy}
0;/r1.2mm/:,
(0,-24)*{\text{(b) $B_2$}},
(-4, 9)*{\text{\small $\begin{bmatrix}1\\0\end{bmatrix}_{1}$}},
(-13, 0)*{\text{\small $\begin{bmatrix}0\\1\end{bmatrix}_{\frac{1}{2}}$}},
(5, -16)*{\text{\small $\begin{bmatrix}1\\1\end{bmatrix}_{\frac{1}{2}}$}},
(11, -15.7)*{\text{\small $\begin{bmatrix}1\\2\end{bmatrix}_{1}$}},
(0,0)="A"
\ar@{-} "A"+(0,0); "A"+(10,0)
\ar@{-} "A"+(0,0); "A"+(0,10)
\ar@{-} "A"+(0,0); "A"+(-10,0)
\ar@{-} "A"+(0,0); "A"+(0,-10)
\ar@{-} "A"+(0,0); "A"+(5,-10)
\ar@{-} "A"+(0,0); "A"+(10,-10)
\end{xy}
\hskip12pt
\begin{xy}
0;/r1.2mm/:,
(0,-24)*{\text{(c) $G_2$}},
(-4, 9)*{\text{\small $\begin{bmatrix}1\\0\end{bmatrix}_{1}$}},
(-13, 0)*{\text{\small $\begin{bmatrix}0\\1\end{bmatrix}_{\frac{1}{3}}$}},
(-1, -16)*{\text{\small $\begin{bmatrix}1\\1\end{bmatrix}_{\frac{1}{3}}$}},
(5, -15.7)*{\text{\small $\begin{bmatrix}2\\3\end{bmatrix}_{1}$}},
(11, -16)*{\text{\small $\begin{bmatrix}1\\2\end{bmatrix}_{\frac{1}{3}}$}},
(17, -15.7)*{\text{\small $\begin{bmatrix}1\\3\end{bmatrix}_{1}$}},
(0,0)="A"
\ar@{-} "A"+(0,0); "A"+(10,0)
\ar@{-} "A"+(0,0); "A"+(0,10)
\ar@{-} "A"+(0,0); "A"+(-10,0)
\ar@{-} "A"+(0,0); "A"+(0,-10)
\ar@{-} "A"+(0,0); "A"+(3.33,-10)
\ar@{-} "A"+(0,0); "A"+(5,-10)
\ar@{-} "A"+(0,0); "A"+(6.66,-10)
\ar@{-} "A"+(0,0); "A"+(10,-10)
\end{xy}
\caption{Rank 2 QCSDs of finite type.}
\label{fig:scat1}
\end{figure}

\par
(a). Type $A_2$. Let $(\d_1,\d_2)=(1,1)$.
Since $\{\bfe_2, \bfe_1\}=1$,
we apply
the pentagon relation \eqref{eq:pent1} with $c=1$,
and we have
\begin{align}
\label{eq:pent0}
\begin{bmatrix}
0\\
1
\end{bmatrix}
_{1,0}
\begin{bmatrix}
1\\
0
\end{bmatrix}
_{1,0}
=
\begin{bmatrix}
1\\
0
\end{bmatrix}
_{1,0}
\begin{bmatrix}
1\\
1
\end{bmatrix}
_{1,0}
\begin{bmatrix}
0\\
1
\end{bmatrix}
_{1,0}.
\end{align}
This equality is naturally interpreted as
a consistency relation of
the QCSD of rank 2 in Figure \ref{fig:scat1} (a).
Namely, it consists of three walls
\begin{align}
(\bfe_1^{\perp}, \Psi_1[\bfe_1])_{\bfe_1},
\quad
(\bfe_2^{\perp}, \Psi_1[\bfe_2])_{\bfe_2},
\quad
(\bbR_{\geq 0} (1,-1), \Psi_1[(1,1)])_{(1,1)}.
\end{align}
The LHS of the equality \eqref{eq:pent0}
is the path-ordered product $\frakp_{\gamma_1, \frakD}$
along $\gamma_1$,
while the RHS is the one along $\gamma_2$
in Figure \ref{fig:scat1}.

\par
(b). Type $B_2$. Let $(\d_1,\d_2)=(1,2)$.
To  use
the pentagon relation \eqref{eq:pent1}  with $c=1$,
we first apply the fission \eqref{eq:fis1} for $\Psi_{1/2}(\bfe_2)$ with $p=2$.
Then apply the pentagon relation repeatedly for adjacent pairs,
and apply the fusion  \eqref{eq:fus1} in the end.
We have
\begin{align}
\label{eq:pent3}
\begin{split}
\begin{bmatrix}
0\\
1
\end{bmatrix}
_{\frac{1}{2},0}
\begin{bmatrix}
1\\
0
\end{bmatrix}
_{1,0}
&=
\begin{bmatrix}
0\\
1
\end{bmatrix}
_{1,-\frac{1}{2}}
\begin{bmatrix}
0\\
1
\end{bmatrix}
_{1,\frac{1}{2}}
\begin{bmatrix}
1\\
0
\end{bmatrix}
_{1,0}
=
\begin{bmatrix}
0\\
1
\end{bmatrix}
_{1,-\frac{1}{2}}
\begin{bmatrix}
1\\
0
\end{bmatrix}
_{1,0}
\begin{bmatrix}
1\\
1
\end{bmatrix}
_{1,\frac{1}{2}}
\begin{bmatrix}
0\\
1
\end{bmatrix}
_{1,\frac{1}{2}}
\\
&=
\begin{bmatrix}
1\\
0
\end{bmatrix}
_{1,0}
\begin{bmatrix}
1\\
1
\end{bmatrix}
_{1,-\frac{1}{2}}
\begin{bmatrix}
0\\
1
\end{bmatrix}
_{1,-\frac{1}{2}}
\begin{bmatrix}
1\\
1
\end{bmatrix}
_{1,\frac{1}{2}}
\begin{bmatrix}
0\\
1
\end{bmatrix}
_{1,\frac{1}{2}}
\\
&=
\begin{bmatrix}
1\\
0
\end{bmatrix}
_{1,0}
\begin{bmatrix}
1\\
1
\end{bmatrix}
_{1,-\frac{1}{2}}
\begin{bmatrix}
1\\
1
\end{bmatrix}
_{1,\frac{1}{2}}
\begin{bmatrix}
1\\
2
\end{bmatrix}
_{1,0}
\begin{bmatrix}
0\\
1
\end{bmatrix}
_{1,-\frac{1}{2}}
\begin{bmatrix}
0\\
1
\end{bmatrix}
_{1,\frac{1}{2}}
\\
&=
\begin{bmatrix}
1\\
0
\end{bmatrix}
_{1,0}
\begin{bmatrix}
1\\
1
\end{bmatrix}
_{\frac{1}{2},0}
\begin{bmatrix}
1\\
2
\end{bmatrix}
_{1,0}
\begin{bmatrix}
0\\
1
\end{bmatrix}
_{\frac{1}{2},0}
.
\end{split}
\end{align}
This equality is naturally interpreted as
a consistency relation of
the QCSD in Figure \ref{fig:scat1} (b),
which consists of four walls
\begin{gather}
(\bfe_1^{\perp}, \Psi_1[\bfe_1])_{\bfe_1},
\quad
(\bfe_2^{\perp}, \Psi_{1/2}[\bfe_2])_{\bfe_2},
\\
(\bbR_{\geq 0} (1,-2), \Psi_{1/2}[(1,1)])_{(1,1)},
\quad
(\bbR_{\geq 0} (1,-1), \Psi_1[(1,2)])_{(1,2)}.
\end{gather}

\par
(c). Type $G_2$. Let $(\d_1,\d_2)=(1,3)$.
In the same way, we obtain
\begin{align}
\label{eq:pent4}
\begin{split}
&\quad\,
\begin{bmatrix}
0\\
1
\end{bmatrix}
_{\frac{1}{3},0}
\begin{bmatrix}
1\\
0
\end{bmatrix}
_{1,0}
=
\begin{bmatrix}
0\\
1
\end{bmatrix}
_{1, -\frac{2}{3}}
\begin{bmatrix}
0\\
1
\end{bmatrix}
_{1, 0}
\begin{bmatrix}
0\\
1
\end{bmatrix}
_{1, \frac{2}{3}}
\begin{bmatrix}
1\\
0
\end{bmatrix}
_{1,0}
\\
&=
\begin{bmatrix}
1\\
0
\end{bmatrix}
_{1,0}
\begin{bmatrix}
1\\
1
\end{bmatrix}
_{1,-\frac{2}{3}}
\begin{bmatrix}
1\\
1
\end{bmatrix}
_{1,0}
\begin{bmatrix}
1\\
1
\end{bmatrix}
_{1,\frac{2}{3}}
\begin{bmatrix}
2\\
3
\end{bmatrix}
_{1,0}
\\
&\qquad\times
\begin{bmatrix}
1\\
2
\end{bmatrix}
_{1,-\frac{2}{3}}
\begin{bmatrix}
1\\
2
\end{bmatrix}
_{1,0}
\begin{bmatrix}
1\\
2
\end{bmatrix}
_{1,\frac{2}{3}}
\begin{bmatrix}
1\\
3
\end{bmatrix}
_{1,0}
\begin{bmatrix}
0\\
1
\end{bmatrix}
_{1, -\frac{2}{3}}
\begin{bmatrix}
0\\
1
\end{bmatrix}
_{1, 0}
\begin{bmatrix}
0\\
1
\end{bmatrix}
_{1, \frac{2}{3}}
\\
&=
\begin{bmatrix}
1\\
0
\end{bmatrix}
_{1,0}
\begin{bmatrix}
1\\
1
\end{bmatrix}
_{\frac{1}{3},0}
\begin{bmatrix}
2\\
3
\end{bmatrix}
_{1,0}
\begin{bmatrix}
1\\
2
\end{bmatrix}
_{\frac{1}{3},0}
\begin{bmatrix}
1\\
3
\end{bmatrix}
_{1,0}
\begin{bmatrix}
0\\
1
\end{bmatrix}
_{\frac{1}{3},0}
.
\end{split}
\end{align}
This equality is naturally interpreted as
a consistency relation of
the QCSD in Figure \ref{fig:scat1} (c),
which consists of six walls
\begin{gather}
(\bfe_1^{\perp}, \Psi_1[\bfe_1])_{\bfe_1},
\quad
(\bfe_2^{\perp}, \Psi_{1/3}[\bfe_2])_{\bfe_2},
\\
(\bbR_{\geq 0} (1,-3), \Psi_{{1/3}}[(1,1)])_{(1,1)},
\quad
(\bbR_{\geq 0} (1,-2), \Psi_1[(2,3)])_{(2,3)},
\\
(\bbR_{\geq 0} (2,-3), \Psi_{{1/3}}[(1,2)])_{(1,2)},
\quad
(\bbR_{\geq 0} (1,-1), \Psi_1[(1,3)])_{(1,3)}.
\end{gather}

Note that, in the RHSs of
\eqref{eq:pent0}, \eqref{eq:pent3}, and \eqref{eq:pent4},
all factors have the form $\Psi_{1/\d(n)}[n]$.

\subsubsection{Affine type}
\label{subsubsec:rank2affine}

Here we  consider the affine type,
where $\d_1 \d_2 = 4$ for $B$ in \eqref{eq:B1}.
  There are two cases.

\begin{figure}
\begin{tikzpicture}[scale=1.4]
\draw(0,0)--(1,0);
\draw(0,0)--(0,1);
\draw(0,0)--(-1,0);
\draw(0,0)--(0,-1);
\draw(0,0)--(0.5,-1);
\draw(0,0)--(0.66,-1);
\draw(0,0)--(0.75,-1);
\draw(0,0)--(0.8,-1);
\draw [very thick] (0,0)--(1,-1);
\draw(0,0)--(1,-0.5);
\draw(0,0)--(1,-0.66);
\draw(0,0)--(1,-0.75);
\draw(0,0)--(1,-0.8);
%
 \node at (0,-1.3){(a) $A_1^{(1)}$};
 \end{tikzpicture}
\hskip60pt
\begin{tikzpicture}[scale=1.4]
\draw(0,0)--(1,0);
\draw(0,0)--(0,1);
\draw(0,0)--(-1,0);
\draw(0,0)--(0,-1);
\draw(0,0)--(0.25,-1);
\draw(0,0)--(0.33,-1);
\draw(0,0)--(0.375,-1);
\draw(0,0)--(0.4,-1);
\draw(0,0)--(0.625,-1);
\draw(0,0)--(0.66,-1);
\draw(0,0)--(0.75,-1);
\draw(0,0)--(1,-1);
\draw [very thick] (0,0)--(0.5,-1);
%
 \node at (0,-1.3){(b) $A_2^{(2)}$};
 \end{tikzpicture}
\vskip-10pt
\caption{Rank 2 QCSDs of affine type.}
\label{fig:scat3}
\end{figure}

(a). Type $A_1^{(1)}$. Let $(\d_1,\d_2)=(2,2)$.
Note that, for any  primitive $n \in N_{\rmpr}^+$,
we have
$\d(n)=2$.
The following description of a  CSD $\frakD_{\fraks}$ is  known
\cite{Gross14,Reineke08, Reading18, Nakanishi22a, Matsushita21}:
The walls of $\frakD_{\fraks}$ are given by
\begin{gather}
\label{eq:a111}
(\bfe_1^{\perp}, \Psi[\bfe_1]^2)_{\bfe_1},
\quad
(\bfe_2^{\perp}, \Psi[\bfe_2]^2)_{\bfe_2},
\\
\label{eq:a112}
(\bbR_{\geq 0} (p,-p-1), \Psi[(p+1,p)]^2)_{(p+1,p)}
\quad
(p\in \bbZ_{>0}),
\\
\label{eq:a113}
(\bbR_{\geq 0} (p+1,-p), \Psi[(p,p+1)]^2)_{(p,p+1)}
\quad
(p\in \bbZ_{>0}),
\\
\label{eq:a114}
(\bbR_{\geq 0} (1,-1), \Psi[2^j\bfn_0]^{2^{2-j}})_{\bfn_0}
\quad
(j \in \bbZ_{\geq 0}),
\end{gather}
where  $\bfn_0=(1,1)$.
Note that $\d((p,p+1))=2$ and $\d(2^j \bfn_0)=2^{1-j}$.
See Figure \ref{fig:scat3} (a).

Again, we apply the fission and the pentagon relation 
to the element $\Psi_{1/2}[\bfe_2] \Psi_{1/2}[\bfe_1]$
and
obtain
the equality
\begin{align}
\label{eq:com2}
\begin{split}
&\quad \,
\begin{bmatrix}
0\\
1
\end{bmatrix}
_{\frac{1}{2},0}
\begin{bmatrix}
1\\
0
\end{bmatrix}
_{\frac{1}{2},0}
=
\begin{bmatrix}
0\\
1
\end{bmatrix}
_{1,-\frac{1}{2}}
\begin{bmatrix}
0\\
1
\end{bmatrix}
_{1,\frac{1}{2}}
\begin{bmatrix}
1\\
0
\end{bmatrix}
_{1,-\frac{1}{2}}
\begin{bmatrix}
1\\
0
\end{bmatrix}
_{1,\frac{1}{2}}
\\
&=
\begin{bmatrix}
1\\
0
\end{bmatrix}
_{1,-\frac{1}{2}}
\begin{bmatrix}
1\\
1
\end{bmatrix}
_{1,-1}
\begin{bmatrix}
1\\
1
\end{bmatrix}
_{1,0}
\begin{bmatrix}
1\\
2
\end{bmatrix}
_{1,-\frac{1}{2}}
\begin{bmatrix}
1\\
0
\end{bmatrix}
_{1,\frac{1}{2}}
\begin{bmatrix}
1\\
1
\end{bmatrix}
_{1,0}
\begin{bmatrix}
1\\
1
\end{bmatrix}
_{1,1}
\begin{bmatrix}
1\\
2
\end{bmatrix}
_{1,\frac{1}{2}}
\begin{bmatrix}
0\\
1
\end{bmatrix}
_{1,-\frac{1}{2}}
\begin{bmatrix}
0\\
1
\end{bmatrix}
_{1,\frac{1}{2}}.
\end{split}
\end{align}
At this moment, the expression in the RHS
is not yet ordered to be presented by
a scattering diagram.
Here, we say that 
the above product is \emph{ordered} (resp. \emph{anti-ordered})
if, for  any adjacent pair $[\bfn']_{a',b'}[\bfn]_{a,b}$,
$\{\bfn',\bfn\}=n'_2 n_1- n'_1 n_2\leq 0$ (resp. $\{\bfn',\bfn\}\geq 0$) holds.
Equivalently, 
if we view $[\bfn]$ as a fraction
$n_1/n_2$,
then, the numbers should be aligned in the decreasing order
(resp. increasing order)
form left to right.
The LHS of \eqref{eq:com2} is anti-ordered.
To make the RHS of \eqref{eq:com2} ordered, we need to interchange
the factors for $\bfn'=(1,2)$ and $\bfn=(1,0)$ in the middle,
where
$\{ \bfn',\bfn\}=2$.
As the lowest approximation, we consider modulo
$G^{>3}$.
Then, $[1,2]_{1,-1/2}$ commutes with other factors, and we have,
modulo $G^{>3}$,
\begin{align}
\label{eq:ordex1}
\begin{bmatrix}
0\\
1
\end{bmatrix}
_{\frac{1}{2},0}
\begin{bmatrix}
1\\
0
\end{bmatrix}
_{\frac{1}{2},0}
&
\equiv
\begin{bmatrix}
1\\
0
\end{bmatrix}
_{\frac{1}{2},0}
\begin{bmatrix}
2\\
1
\end{bmatrix}
_{\frac{1}{2},0}
\begin{bmatrix}
1\\
1
\end{bmatrix}
_{\frac{1}{2},-\frac{1}{2}}
\begin{bmatrix}
1\\
1
\end{bmatrix}
_{\frac{1}{2},\frac{1}{2}}
\begin{bmatrix}
1\\
2
\end{bmatrix}
_{\frac{1}{2},0}
\begin{bmatrix}
0\\
1
\end{bmatrix}
_{\frac{1}{2},0}.
\end{align}

To proceed to higher degree, we
now apply the pentagon relation  \eqref{eq:pent1}
with $c=2$ to the pair $[1,2]_{1,-1/2}$ and $[1,0]_{1,1/2}$
in \eqref{eq:com2}.
Then, we have
\begin{align}
\label{eq:com3}
\begin{split}
&\quad \,
\begin{bmatrix}
1\\
2
\end{bmatrix}
_{1,-\frac{1}{2}}
\begin{bmatrix}
1\\
0
\end{bmatrix}
_{1,\frac{1}{2}}
=
\begin{bmatrix}
1\\
2
\end{bmatrix}
_{2,-\frac{3}{2}}
\begin{bmatrix}
1\\
2
\end{bmatrix}
_{2,\frac{1}{2}}
\begin{bmatrix}
1\\
0
\end{bmatrix}
_{2,-\frac{1}{2}}
\begin{bmatrix}
1\\
0
\end{bmatrix}
_{2,\frac{3}{2}}
\\
&=
\begin{bmatrix}
1\\
0
\end{bmatrix}
_{2,-\frac{1}{2}}
\begin{bmatrix}
2\\
2
\end{bmatrix}
_{2,-2}
\begin{bmatrix}
2\\
2
\end{bmatrix}
_{2,0}
\begin{bmatrix}
3\\
4
\end{bmatrix}
_{2,-\frac{3}{2}}
\begin{bmatrix}
1\\
0
\end{bmatrix}
_{2,\frac{3}{2}}
\begin{bmatrix}
2\\
2
\end{bmatrix}
_{2,0}
\begin{bmatrix}
2\\
2
\end{bmatrix}
_{2,2}
\begin{bmatrix}
3\\
4
\end{bmatrix}
_{2,\frac{1}{2}}
\begin{bmatrix}
1\\
2
\end{bmatrix}
_{2,-\frac{3}{2}}
\begin{bmatrix}
1\\
2
\end{bmatrix}
_{2,\frac{1}{2}}.
\end{split}
\end{align}
This is parallel to \eqref{eq:com2} but with different quantum data.
As the next approximation, we consider modulo
$G^{>7}$.
Then, $[3,4]_{2,-3/2}$ commutes with other factors, and we have,
modulo $G^{>7}$,
\begin{align}
\label{eq:fact1}
\begin{bmatrix}
1\\
2
\end{bmatrix}
_{1,-\frac{1}{2}}
\begin{bmatrix}
1\\
0
\end{bmatrix}
_{1,\frac{1}{2}}
&
\equiv
\begin{bmatrix}
1\\
0
\end{bmatrix}
_{1,\frac{1}{2}}
\begin{bmatrix}
3\\
2
\end{bmatrix}
_{1,\frac{1}{2}}
\begin{bmatrix}
2\\
2
\end{bmatrix}
_{1,-1}
\begin{bmatrix}
2\\
2
\end{bmatrix}
_{1,1}
\begin{bmatrix}
3\\
4
\end{bmatrix}
_{1,-\frac{1}{2}}
\begin{bmatrix}
1\\
2
\end{bmatrix}
_{1,-\frac{1}{2}}
.
\end{align}
Then, we plug it into
\eqref{eq:com2},
and apply the pentagon relation,
and
we have,
modulo $G^{>7}$,
\begin{align}
\label{eq:ordex2}
\begin{split}
\begin{bmatrix}
0\\
1
\end{bmatrix}
_{\frac{1}{2},0}
\begin{bmatrix}
1\\
0
\end{bmatrix}
_{\frac{1}{2},0}
&
\equiv
\begin{bmatrix}
1\\
0
\end{bmatrix}
_{\frac{1}{2},0}
\begin{bmatrix}
2\\
1
\end{bmatrix}
_{\frac{1}{2},0}
\begin{bmatrix}
3\\
2
\end{bmatrix}
_{\frac{1}{2},0}
\begin{bmatrix}
4\\
3
\end{bmatrix}
_{\frac{1}{2},0}
\begin{bmatrix}
1\\
1
\end{bmatrix}
_{\frac{1}{2},-\frac{1}{2}}
\begin{bmatrix}
1\\
1
\end{bmatrix}
_{\frac{1}{2},\frac{1}{2}}
\\
&\quad\
\times
\begin{bmatrix}
2\\
2
\end{bmatrix}
_{1,-1}
\begin{bmatrix}
2\\
2
\end{bmatrix}
_{1,1}
\begin{bmatrix}
3\\
4
\end{bmatrix}
_{\frac{1}{2},0}
\begin{bmatrix}
2\\
3
\end{bmatrix}
_{\frac{1}{2},0}
\begin{bmatrix}
1\\
2
\end{bmatrix}
_{\frac{1}{2},0}
\begin{bmatrix}
0\\
1
\end{bmatrix}
_{\frac{1}{2},0}.
\end{split}
\end{align}
By continuing the procedure modulo  $G^{>2^{\l}-1}$,
one can naturally guess that
 the relation converges to the following one:
\begin{align}
\label{eq:a115}
\begin{split}
\begin{bmatrix}
0\\
1
\end{bmatrix}
_{\frac{1}{2},0}
\begin{bmatrix}
1\\
0
\end{bmatrix}
_{\frac{1}{2},0}
&
=
\begin{bmatrix}
1\\
0
\end{bmatrix}
_{\frac{1}{2},0}
\begin{bmatrix}
2\\
1
\end{bmatrix}
_{\frac{1}{2},0}
\begin{bmatrix}
3\\
2
\end{bmatrix}
_{\frac{1}{2},0}
\cdots
\Biggl(
\prod_{j=0}^{\infty}
\begin{bmatrix}
2^j\\
2^j
\end{bmatrix}
_{2^{j-1},-2^{j-1}}
\begin{bmatrix}
2^j\\
2^j
\end{bmatrix}
_{2^{j-1},2^{j-1}}
\Biggr)
\\
&\qquad \times
\cdots
\begin{bmatrix}
2\\
3
\end{bmatrix}
_{\frac{1}{2},0}
\begin{bmatrix}
1\\
2
\end{bmatrix}
_{\frac{1}{2},0}
\begin{bmatrix}
0\\
1
\end{bmatrix}
_{\frac{1}{2},0}.
\end{split}
\end{align}
Indeed,
one can prove the relation \eqref{eq:a115} completely
 based on the pentagon relation
 by modifying the proof of  \cite{Matsushita21} for the classical case.
 The detail will be found in Appendix \ref{sec:derivation}.
This is the simplest example of consistency relations
involving infinite products.
The formula  appeared in \cite[Eq.~(1.3)]{Dimofte09} without proof
in an alternative setting of quantum dilogarithms.
See also an alternative derivation in view of the quantum affine algebra of type $A_1^{(1)}$
\cite[Theorem 5.1]{Sugawara22}.

\smallskip

(b). Type $A_2^{(2)}$. Let $(\d_1,\d_2)=(1,4)$.
The situation is in parallel with $A_1^{(1)}$ though a little more complicated.
The following description of the  CSD $\frakD_{\fraks}$ is  known \cite{Reading18, Nakanishi22a, Matsushita21}:
The walls of $\frakD_{\fraks}$ are given by
\begin{gather}
\label{eq:a221}
(\bfe_1^{\perp}, \Psi[\bfe_1])_{\bfe_1},
\quad
(\bfe_2^{\perp}, \Psi[\bfe_2]^4)_{\bfe_2},
\\
\label{eq:a222}
(\bbR_{\geq 0} (p,-2p-1), \Psi[(2p+1,4p)])_{(2p+1,4p)}
\quad
(p\in \bbZ_{>0}),
\\
\label{eq:a223}
(\bbR_{\geq 0} (2p-1,-4p), \Psi[(p,2p-1)]^4)_{(p,2p-1)}
\quad
(p\in \bbZ_{>0}),
\\
\label{eq:a224}
(\bbR_{\geq 0} (p,-2p+1), \Psi[(2p-1,4p)])_{(2p-1,4p)}
\quad
(p\in \bbZ_{>0}),
\\
\label{eq:a225}
(\bbR_{\geq 0} (2p+1,-4p), \Psi[(p,2p+1)]^4)_{(p,2p+1)}
\quad
(p\in \bbZ_{>0}),
\\
\label{eq:a226}
 (\bbR_{\geq 0} (1,-2), \Psi[\bfn_0]^6)_{\bfn_0},
 \quad
 (\bbR_{\geq 0} (1,-2), \Psi[2^j \bfn_0]^{2^{2-j}})_{\bfn_0}
 \quad
 (j\in \bbZ_{>0}),
\end{gather}
where  $\bfn_0=(1,2)$.
Note that $\d(2^j \bfn_0)=2^{1-j}$.
See Figure \ref{fig:scat3} (b).

Again, we apply the fission and  the pentagon relation 
to the  element $\Psi_{1/4}[\bfe_2] \Psi_{1}[\bfe_1]$
and
obtain
the equality
\begin{align}
\label{3eq:com4}
\begin{split}
&\quad\,
\begin{bmatrix}
0\\
1
\end{bmatrix}
_{\frac{1}{4},0}
\begin{bmatrix}
1\\
0
\end{bmatrix}
_{1,0}
=
\begin{bmatrix}
0\\
1
\end{bmatrix}
_{1,-\frac{3}{4}}
\begin{bmatrix}
0\\
1
\end{bmatrix}
_{1,-\frac{1}{4}}
\begin{bmatrix}
0\\
1
\end{bmatrix}
_{1,-\frac{1}{4}}
\begin{bmatrix}
0\\
1
\end{bmatrix}
_{1,\frac{3}{4}}
\begin{bmatrix}
1\\
0
\end{bmatrix}
_{1,0}
\\
&=
\begin{bmatrix}
1\\
0
\end{bmatrix}
_{1,0}
\begin{bmatrix}
1\\
1
\end{bmatrix}
_{1,-\frac{3}{4}}
\begin{bmatrix}
1\\
1
\end{bmatrix}
_{1,-\frac{1}{4}}
\begin{bmatrix}
1\\
1
\end{bmatrix}
_{1,\frac{1}{4}}
\begin{bmatrix}
2\\
3
\end{bmatrix}
_{1,-\frac{3}{4}}
\begin{bmatrix}
1\\
2
\end{bmatrix}
_{1,-1}
\begin{bmatrix}
1\\
2
\end{bmatrix}
_{1,-\frac{1}{2}}
\begin{bmatrix}
1\\
2
\end{bmatrix}
_{1,0}
\begin{bmatrix}
1\\
3
\end{bmatrix}
_{1,-\frac{3}{4}}
\begin{bmatrix}
1\\
1
\end{bmatrix}
_{1,\frac{3}{4}}
\\
&\qquad
\times
\begin{bmatrix}
1\\
2
\end{bmatrix}
_{1,0}
\begin{bmatrix}
1\\
2
\end{bmatrix}
_{1,\frac{1}{2}}
\begin{bmatrix}
1\\
2
\end{bmatrix}
_{1,1}
\begin{bmatrix}
2\\
5
\end{bmatrix}
_{1,\frac{3}{4}}
\begin{bmatrix}
1\\
3
\end{bmatrix}
_{1,-\frac{1}{4}}
\begin{bmatrix}
1\\
3
\end{bmatrix}
_{1,\frac{1}{4}}
\begin{bmatrix}
1\\
3
\end{bmatrix}
_{1,\frac{3}{4}}
\begin{bmatrix}
1\\
4
\end{bmatrix}
_{1,0}
\begin{bmatrix}
0\\
1
\end{bmatrix}
_{\frac{1}{4},0}
.
\end{split}
\end{align}
Again, this
is not yet ordered,
and
to order it, we need to interchange
$[1,3]_{1,-3/4}$ and $[1,1]_{1,3/4}$ in the middle.
As the lowest approximation, we consider modulo
$G^{>5}$.
Then, $\Psi_{1,-3/4}[1,3]$ commutes with other factors, and we have,
modulo $G^{>5}$,
\begin{align}
\begin{split}
& \quad\,
\begin{bmatrix}
0\\
1
\end{bmatrix}
_{\frac{1}{4},0}
\begin{bmatrix}
1\\
0
\end{bmatrix}
_{1,0}
\equiv
\begin{bmatrix}
1\\
0
\end{bmatrix}
_{1,0}
\begin{bmatrix}
1\\
1
\end{bmatrix}
_{\frac{1}{4},0}
\begin{bmatrix}
2\\
3
\end{bmatrix}
_{\frac{1}{4},0}
\begin{bmatrix}
1\\
2
\end{bmatrix}
_{\frac{1}{2},-\frac{1}{2}}
\begin{bmatrix}
1\\
2
\end{bmatrix}
_{\frac{1}{2},0}
\begin{bmatrix}
1\\
2
\end{bmatrix}
_{\frac{1}{2},\frac{1}{2}}
\begin{bmatrix}
1\\
3
\end{bmatrix}
_{\frac{1}{4},0}
\begin{bmatrix}
1\\
4
\end{bmatrix}
_{1,0}
\begin{bmatrix}
0\\
1
\end{bmatrix}
_{\frac{1}{4},0}
.
\end{split}
\end{align}

To proceed to higher degree,
we do
in the same way as before.
By applying the fission and the pentagon relation  \eqref{eq:pent1} with $c=2$, 
we have
 \begin{align}
\begin{split}
&\quad \,
\begin{bmatrix}
1\\
3
\end{bmatrix}
_{1,-\frac{3}{4}}
\begin{bmatrix}
1\\
1
\end{bmatrix}
_{1,\frac{3}{4}}
=
\begin{bmatrix}
1\\
3
\end{bmatrix}
_{2,-\frac{7}{4}}
\begin{bmatrix}
1\\
3
\end{bmatrix}
_{2,\frac{1}{4}}
\begin{bmatrix}
1\\
1
\end{bmatrix}
_{2,-\frac{1}{4}}
\begin{bmatrix}
1\\
1
\end{bmatrix}
_{2,\frac{7}{4}}
\\
&=
\begin{bmatrix}
1\\
1
\end{bmatrix}
_{2,-\frac{1}{4}}
\begin{bmatrix}
2\\
4
\end{bmatrix}
_{2,-2}
\begin{bmatrix}
2\\
4
\end{bmatrix}
_{2,0}
\begin{bmatrix}
3\\
7
\end{bmatrix}
_{2,-\frac{7}{4}}
\begin{bmatrix}
1\\
1
\end{bmatrix}
_{2,\frac{7}{4}}
\begin{bmatrix}
2\\
4
\end{bmatrix}
_{2,0}
\begin{bmatrix}
2\\
4
\end{bmatrix}
_{2,2}
\begin{bmatrix}
3\\
7
\end{bmatrix}
_{2,\frac{1}{4}}
\begin{bmatrix}
1\\
3
\end{bmatrix}
_{2,-\frac{7}{4}}
\begin{bmatrix}
1\\
3
\end{bmatrix}
_{2,\frac{1}{4}}.
\end{split}
\end{align}
This parallel to
 \eqref{eq:com2}, but with different quantum data.
As the next approximation, we consider modulo
$G^{>11}$.
Then, $[3,7]_{2,-7/4}$ commutes with other factors, and we have,
modulo $G^{>11}$,
\begin{align}
\begin{bmatrix}
1\\
3
\end{bmatrix}
_{1,-\frac{3}{4}}
\begin{bmatrix}
1\\
1
\end{bmatrix}
_{1,\frac{3}{4}}
&
\equiv
\begin{bmatrix}
1\\
1
\end{bmatrix}
_{1,\frac{3}{4}}
\begin{bmatrix}
3\\
5
\end{bmatrix}
_{1,\frac{3}{4}}
\begin{bmatrix}
2\\
4
\end{bmatrix}
_{1,-1}
\begin{bmatrix}
2\\
4
\end{bmatrix}
_{1,1}
\begin{bmatrix}
3\\
7
\end{bmatrix}
_{1,-\frac{3}{4}}
\begin{bmatrix}
1\\
3
\end{bmatrix}
_{1,-\frac{3}{4}}
.
\end{align}
Then, we plug it into
\eqref{3eq:com4},
and apply the pentagon relation,
and
we have,
modulo $G^{>11}$,
\begin{align}
\begin{split}
\begin{bmatrix}
0\\
1
\end{bmatrix}
_{\frac{1}{4},0}
\begin{bmatrix}
1\\
0
\end{bmatrix}
_{1,0}
%
&
\equiv
\begin{bmatrix}
1\\
0
\end{bmatrix}
_{1,0}
\begin{bmatrix}
1\\
1
\end{bmatrix}
_{\frac{1}{4},0}
\begin{bmatrix}
3\\
4
\end{bmatrix}
_{1,0}
\begin{bmatrix}
2\\
3
\end{bmatrix}
_{\frac{1}{4},0}
\begin{bmatrix}
3\\
5
\end{bmatrix}
_{\frac{1}{4},0}
\begin{bmatrix}
4\\
7
\end{bmatrix}
_{\frac{1}{4},0}
\begin{bmatrix}
1\\
2
\end{bmatrix}
_{\frac{1}{2},-\frac{1}{2}}
\begin{bmatrix}
1\\
2
\end{bmatrix}
_{\frac{1}{2},0}
\begin{bmatrix}
1\\
2
\end{bmatrix}
_{\frac{1}{2},\frac{1}{2}}
\\
&\qquad\times
\begin{bmatrix}
2\\
4
\end{bmatrix}
_{1,-1}
\begin{bmatrix}
2\\
4
\end{bmatrix}
_{1,1}
\begin{bmatrix}
3\\
7
\end{bmatrix}
_{\frac{1}{4},0}
\begin{bmatrix}
2\\
5
\end{bmatrix}
_{\frac{1}{4},0}
\begin{bmatrix}
3\\
8
\end{bmatrix}
_{1,0}
\begin{bmatrix}
1\\
3
\end{bmatrix}
_{\frac{1}{4},0}
\begin{bmatrix}
1\\
4
\end{bmatrix}
_{1,0}
\begin{bmatrix}
0\\
1
\end{bmatrix}
_{\frac{1}{4},0}.
\end{split}
\end{align}
By continuing the procedure modulo  $G^{>2^{\l}-1}$,
one can naturally guess that
 the relation converges to the following one:
\begin{align}
\label{eq:a227}
\begin{split}
\begin{bmatrix}
0\\
1
\end{bmatrix}
_{\frac{1}{4},0}
\begin{bmatrix}
1\\
0
\end{bmatrix}
_{1,0}
&=
\begin{bmatrix}
1\\
0
\end{bmatrix}
_{1,0}
\begin{bmatrix}
1\\
1
\end{bmatrix}
_{\frac{1}{4},0}
\begin{bmatrix}
3\\
4
\end{bmatrix}
_{1,0}
\begin{bmatrix}
2\\
3
\end{bmatrix}
_{\frac{1}{4},0}
\begin{bmatrix}
5\\
8
\end{bmatrix}
_{1,0}
\begin{bmatrix}
3\\
5
\end{bmatrix}
_{\frac{1}{4},0}
\cdots
\\
&
\qquad
\times
\Biggl(
\begin{bmatrix}
1\\
2
\end{bmatrix}
_{\frac{1}{2},0}
\prod_{j=0}^{\infty}
\begin{bmatrix}
 2^j\\
 2^{j+1}
\end{bmatrix}
_{2^{j-1}, -2^{j-1}}
\begin{bmatrix}
 2^j\\
 2^{j+1}
\end{bmatrix}
_{2^{j-1}, 2^{j-1}}
\Biggr)
\\
&
\qquad 
\times
\cdots
\begin{bmatrix}
5\\
12
\end{bmatrix}
_{1,0}
\begin{bmatrix}
2\\
5
\end{bmatrix}
_{\frac{1}{4},0}
\begin{bmatrix}
3\\
8
\end{bmatrix}
_{1,0}
\begin{bmatrix}
1\\
3
\end{bmatrix}
_{\frac{1}{4},0}
\begin{bmatrix}
1\\
4
\end{bmatrix}
_{1,0}
\begin{bmatrix}
0\\
1
\end{bmatrix}
_{\frac{1}{4},0}
.
\end{split}
\end{align}
Again,
one can prove the relation \eqref{eq:a227} 
completely
 based on the pentagon relation
  by modifying the proof of  \cite{Matsushita21} for the classical case.
   The detail will be found in Appendix \ref{sec:derivation}.

Note that, in the RHSs of \eqref{eq:a115} and 
 \eqref{eq:a227},
 all factors have the form
$\Psi_{1/\d(n),b}[n]$.

\section{Principal $x$-representation}
\label{sec:principal1}

We introduce another representation of $G$ which we call
the \emph{principal $x$-representation}.
It is closely related to the quantization of cluster variables ($x$-variables) in \cite{Berenstein05b}.
Then, we obtain the reduction property of wall elements for a QCSD,
which is important in our application of the pentagon relation to the positivity problem.

\subsection{Principal $x$-representation}
Following  \cite{Gross14},
we introduce
\begin{align}
\tilde N= N \oplus M^{\circ},
\quad
\tilde M^{\circ}=M^{\circ} \oplus N,
\end{align}
which are ``dual'' to each other twisted by $\d_1,\dots,\d_r$.
An element of $\tilde N$ is denoted by
$\tilde n = (n,m)$ ($n\in N$, $m\in M^{\circ}$).
Similarly,
an element of $\tilde M^{\circ}$ is denoted by
$\tilde m = (m,n)$ ($m\in M^{\circ}$, $n\in N$).
Below we use  the notation such as $\tilde n_1 = (n_1,m_1)$,
$\tilde m' = (m',n')$ without explanation.
Meanwhile, we also use the internal-sum notation.
For example, for $n\in N$, 
we write $n=(n,0)\in \tilde N$ and
$n=(0,n)\in \tilde M^{\circ}$.
Both spaces  $\tilde N$ and $\tilde M^{\circ}$ are equipped with  
(essentially common) bilinear forms
\begin{align}
\{\tilde n, \tilde n'\}_{\tilde N}&:=
\{n, n'\} + \langle n',m\rangle - \langle n, m'\rangle,
\\
\label{eq:bi1}
\{\tilde m, \tilde m'\}_{\tilde M^{\circ}}&:=
- \{n, n'\} - \langle n',m\rangle + \langle n, m'\rangle,
\end{align}
where $\{\cdot, \cdot\}$ and $\langle\cdot, \cdot\rangle$
in the RHSs
are the ones in Section \ref{subsec:structure1}.
We also have the canonical paring
$\langle\cdot, \cdot\rangle: \tilde N \times \tilde M^{\circ}
\rightarrow (1/\delta_0)\bbZ$ defined by
\begin{align}
\langle (n,m),(m',n')\rangle:=
\langle n,m'\rangle
+
\langle n',m\rangle.
\end{align}

Let $\tilde p^*$ be the group homomorphism defined by
\begin{align}
\begin{matrix}
\tilde p^*\colon & N &\rightarrow & \tilde M^{\circ}\\
& n& \mapsto & \{\cdot, n \}_{\tilde N}
\end{matrix},
\quad
\{(n',m'),n\}_{\tilde N}=\{n',n\}+\langle n, m'\rangle.
\end{align}
Note that the map $\tilde p^*$ is injective.
Let $p^{*} \colon N \rightarrow   M^{\circ}$ be the one in Section
\ref{subsec:quantum2}.
Then, we have
\begin{align}
\label{eq:p11}
\tilde p^*(n)=(p^*(n), n)
\in \tilde M^{\circ}.
\end{align}
The following duality relation holds:
\begin{align}
\label{eq:dual1}
 \{\tilde p^*(n), \tilde m \}_{\tilde M^{\circ}}=\langle n, \tilde m \rangle.
\end{align}
Indeed, by \eqref{eq:p11},
\begin{align}
\begin{split}
 \{\tilde p^*(n), \tilde m \}_{\tilde M^{\circ}}
&=
\{(p^*(n),n),(m',n')\}_{\tilde M^{\circ}}
\\
&=- \{n,n'\}
-\langle n', p^*(n)\rangle
+\langle n,m'\rangle
\\
&=\langle n,m'\rangle
=\langle n,\tilde m\rangle.
\end{split}
\end{align}

\begin{rem}
\label{rem:compatible1}
The  bilinear form $\{\cdot, \cdot\}_{\tilde M^{\circ}}$  in \eqref{eq:bi1}  has the representation matrix
with respect to the basis $f_1$, \dots, $f_r$, $e_1$, \dots, $e_r$
of $\tilde M^{\circ}$ as
\begin{align}
\label{eq:Lambda1}
\Lambda:=
\begin{pmatrix}
O & - D\\
D & - DB
\end{pmatrix},
\quad
D=\mathrm{diag}(\d_1^{-1},\cdots,\d_r^{-1}).
\end{align}
Let $\tilde B$ be the principal extension of the matrix $B$,
\begin{align}
\tilde B =
\begin{pmatrix}
B\\
I
\end{pmatrix}.
\end{align}
They satisfy the relation
\begin{align}
\label{eq:dual2}
-\Lambda \tilde B =
\begin{pmatrix}
D\\
O
\end{pmatrix},
\end{align}
which is equivalent to the duality relation \eqref{eq:dual1}.
Such a pair $(\Lambda, \tilde B)$ is called a \emph{compatible pair} in \cite{Berenstein05b}.
 \end{rem}

Following \cite{Gross14},
let $\tilde P\subset \tilde M^{\circ}$ be a (not unique) monoid satisfying the following conditions:
\begin{itemize}
\item[(i).]
$\tilde P=\sigma \cap \tilde M^{\circ}$,
where $\sigma$ is a $2r$-dimensional strongly convex cone in $\tilde M^{\circ}_{\bbR}$.
\item[(ii).]
$\tilde p^*(e_1)$, \dots, $\tilde p^*(e_r)\in \tilde P$.
\item[(iii).]
$f_1$, \dots, $f_r\in \tilde P$.
\end{itemize}
For example, we take $\sigma$ as the cone generated by $f_1$, \dots, $f_r$, $\tilde p^*(e_1)$, \dots, $\tilde p^*(e_r)$,
which are a basis of $\tilde M^{\circ}$.
(In \cite{Gross14}, the condition (iii) was not assumed.)

Now we are ready to define the \emph{(quantum) principal $x$-representation} of $G$,
which is defined in a parallel way to the $y$-representation.

Let $x$ be a formal symbol.
Let $\bbQ(q^{ 1/\delta_0})[x]_q$ be the noncommutative and associative algebra over
$\bbQ(q^{ 1/\delta_0})$ with generators $x^{\tilde m}$ ($ \tilde m\in \tilde P)$
and the $q$-commutative relations
\begin{align}
\label{eq:P11}
x^{\tilde m}x^{\tilde m'}=
q^{\{\tilde m, \tilde m'\}_{\tilde M^{\circ}}}
x^{\tilde m + \tilde m'}.
\end{align}
Let $\calR_q(x)$  be the completion of $\bbQ(q^{ 1/\delta_0})[x]_q$
with respect to the maximal ideal  generated by 
$x^{\tilde m}$'s ($\tilde m \in \tilde P\setminus \{0\})$.
Any element of $\calR_q(x)$ is expressed as an infinite sum
\begin{align}
\label{eq:cxm1}
\sum_{\tilde m\in \tilde P}
c_{\tilde m} x^{\tilde m}
\quad
(c_{\tilde m} \in \bbQ(q^{ 1/\delta_0})).
\end{align}

We define the action of  $\widehat\frakg$ on $\calR_q(x)$ by
\begin{align}
\label{eq:Xact4}
X_n(x^{\tilde m}):= &\ 
[\langle n,\tilde m \rangle]_q
x^{\tilde m +\tilde p^*(n)}\\
=&\ \frac{q^{2\langle n,\tilde m \rangle}-1}{q-q^{-1}}x^{\tilde m} x^{\tilde p^*(n)},
\end{align}
where in the second equality we used
the duality relation \eqref{eq:dual1}.
It is easy to check that this is an action of $\widehat\frakg$ and also a derivation.
Thus,
it induces the action  of $G$ on  $\calR_q(x)$ defined by
\begin{align}
\label{eq:Xaction3}
(\exp X)(x^{\tilde m})=
\sum_{j=0}^{\infty}
\frac{1}{j!} X^j(x^{\tilde m})
\quad
(X\in \widehat\frakg).
\end{align}
Moreover, $\exp X$ is an algebra automorphism.
We call the resulting representation
 $\rho_x^{\rmpr}: G\rightarrow \mathrm{Aut}(\calR_q(x))$
 the \emph{(quantum) principal $x$-representation} of $G$.
It is always a faithful representation,
even when $\{\cdot, \cdot\}$ is degenerate.

\subsection{Action of dilogarithm elements}
Let us present an analogous result in
Section \ref{subsec:quantum1} for the principal $x$-representation.

Any $\bbQ(q^{1/\delta_0})$-power of a dilogarithm element  $\Psi_{a,b}[n]^c$ ($c\in \bbQ(q^{1/\delta_0}))$ acts on  $\calR_q(x)$ under $\rho_x^{\mathrm{pr}}$ as
\begin{align}
\label{eq:xmut1}
\begin{split}
\Psi_{a,b}[n]^c(x^{\tilde m})
&= \exp\Biggl(\, c \sum_{j=1}^{\infty}
 \frac{(-1)^{j+1}}{j[ja]_q} q^{jb} X_{jn}\Biggr)(x^{\tilde m})\\
&=x^{\tilde m} \exp\Biggl(
 c \sum_{j=1}^{\infty} 
 \frac{q^{2j\langle n,\tilde m \rangle}-1}{q^{2ja}-1} \frac{(-1)^{j+1}}{j }q^{ja}q^{jb} 
 x^{ j \tilde p^*(n)}
 \Biggr).
 \end{split}
 \end{align}
 Let us consider a parallel situation to \eqref{eq:qq0}
 and \eqref{eq:ymut2}.
Fix $n$ in \eqref{eq:xmut1},
and
let $a=1/s\d(n)$, where $s\in \bbZ_{>0}$ is a divisor of $\delta_0/\d(n)$.
Then, $\alpha:=\langle s \d(n)n,\tilde m \rangle$ is an integer,
and 
we have
 \begin{align}
 \label{eq:q11}
 \begin{split}
  \frac{q^{2j\langle n,\tilde m \rangle}-1}{q^{2ja}-1} 
  &=
  \frac{q^{2\alpha j/s\d(n)}-1}{q^{2j/s\d(n)}-1} 
  =
  \begin{cases}
\displaystyle
\sum_{p=0}^{\alpha-1}
q^{2jp/s\d(n)}
& \alpha > 0,
\\
0
&
\alpha=0,
\\
\displaystyle
-
\sum_{p=1}^{-\alpha}
q^{-2jp/s\d(n)}
& \alpha< 0.
\end{cases}
\end{split}
 \end{align}
Thus, we obtain
 \begin{align}
 \label{eq:xmut2}
 \begin{split}
 &\quad \
 \Psi_{1/s\d(n),b}[n](x^{\tilde m})   
=
\begin{cases}
\displaystyle
x^{\tilde m}\prod_{p=1}^{\alpha}
(1+q^{(2p-1)/s\d(n)}q^b  x^{ \tilde p^*(n)})
& \alpha> 0,
\\
x^{\tilde m}
&\alpha=0,
\\
\displaystyle
x^{\tilde m}\prod_{p=1}^{-\alpha}
(1+q^{-(2p-1)/s\d(n)}q^b  x^{ \tilde p^*(n)})^{-1}
& \alpha< 0.
\end{cases}
\end{split}
\end{align}
Observe that
this is indeed the automorphism part of the Fock-Goncharov decomposition
of mutations of the quantum  $x$-variables 
with principal coefficients in \cite{Berenstein05b}.
See also \cite{Mandel15,Davison19, Cheung20}
for an alternative approach, where the same representation is described through the {adjoint action
of the quantum dilogarithm} on $\calR_q(x)$ as in Remark \ref{rem:adjoint1}.

\subsection{Reduction of wall elements in a QCSD}
\label{subsec:reduction1}
Let $\calS_q(x)$ be the subset of $\calR_q(x)$
consisting of the elements
such that the coefficients $c_{\tilde m}$ in
\eqref{eq:cxm1} are  in $ \bbZ[q^{\pm1/\delta_0}]$.
Then, $\calS_q(x)$ is an algebra over $ \bbZ[q^{\pm1/\delta_0}]$ by the same relation
\eqref{eq:P11}.
The following lemma is a quantum analog of
the one in \cite[Appendix C.3 Step II]{Gross14},
\cite[Lemma III.5.8]{Nakanishi22a}.

\begin{lem}
\label{lem:actZ1}
A $\bbQ(q^{1/\delta_0})$-power of a
quantum dilogarithm element $\Psi_{a,b}[n]^c$
$(c\in \bbQ(q^{1/\delta_0}))$
acts on $\calS_q(x)$ 
if and only if $a=1/s\d(n)$ for a divisor $s$ of  $\delta_0/\d(n)$
and $c\in \bbZ[q^{\pm1/\delta_0}]$.
\end{lem}
\begin{proof}
First,
we prove the if-part.
By the assumption $c\in \bbZ[q^{\pm1/\delta_0}]$,
$\Psi_{1/s\d(n),b}[n]^c$ is decomposed into a product of 
$\Psi_{1/s\d(n),b_i}[n]^{c_i}$ ($b_i\in (1/\delta_0)\bbZ$, $c_i\in \bbZ$).
By \eqref{eq:xmut2}, each factor acts on $\calS_q(x)$.
Thus, $\Psi_{a,b}[n]^c$ acts on $\calS_q(x)$.
Next,
we prove the only-if-part.
Assume that $\Psi_{a,b}[n]^c$ acts on $\calS_q(x)$.
Then,
by \eqref{eq:xmut1},
we have a necessary condition
\begin{align}
\label{eq:jterm1}
 c 
 \frac{q^{2j\langle n,\tilde m \rangle}-1}{q^{2ja}-1} \frac{(-1)^{j+1}}{j }q^{ja}q^{jb} 
 \in \bbQ[q^{\pm1/\delta_0}]
 \quad
 (j\in \bbZ_{>0}).
\end{align}
Let $a=1/s\d(n)$, where $s$ is a positive rational number.
Then,
\eqref{eq:jterm1} implies the condition
$s \langle \d(n)n,\tilde m \rangle \in \bbZ$ for any $\tilde m\in \tilde{P}$.
By the assumption for $\tilde P$, $\tilde P$ contains a basis of $\tilde{M}^{\circ}$.
It follows that
$s \langle \d(n)n,\tilde m \rangle \in \bbZ$ for any $\tilde m\in \tilde{M}^{\circ}$.
Due to the definition of $\delta(n)$, there is some $\tilde m\in \tilde{M}^{\circ}$
such that $\langle \d(n)n,\tilde m \rangle=1$. Thus, we have $s\in \bbZ$.
On the other hand,
let $t=\d_0/\d(n)\in \bbZ_{>0}$.
Then, $a=t/s\d_0$.
Since we require that $a \in (1/\d_0)\bbZ_{>0}$,
$s$ is a divisor of $t$.
This finishes the condition for $a$.
Meanwhile, the condition \eqref{eq:jterm1} also implies $c\in\bbQ[q^{\pm1/\delta_0}]$.
Then, 
$\Psi_{a,b}[n]^c$ is decomposed into a product of 
$\Psi_{1/s\d(n),b_i}[n]^{c_i}$ ($b_i\in (1/\delta_0)\bbZ$, $c_i\in \bbQ$).
By \eqref{eq:xmut2}, we have $c_i\in \bbZ$.
Thus, we have $c\in\bbZ[q^{\pm1/\delta_0}]$.
\end{proof}

Below, we only use the if-part of Lemma \ref{lem:actZ1}.

\begin{prop}
\label{prop:psi1}
Any QCSD $\frakD_{\fraks}^q$
is equivalent to a consistent scattering diagram whose wall elements have 
the form
\begin{align}
\label{eq:psipm1}
 \Psi_{1/\d(n),b}[n]^{ c}
 \quad
 (n\in N^+,\ b \in (1/\delta_0)\bbZ,\ c\in \bbZ).
 \end{align}
\end{prop}
\begin{proof}
We follow Step 2 in the proof of \cite[Prop.~C.13]{Gross14}.
This is proved inductively on the degree of $n$, by considering the consistency relations around the perpendicular joints.
Here, we only describe the key point of the inductive step.
Suppose that a scattering diagram $\frakD_{\ell}$ has the
walls with wall elements of the form  \eqref{eq:psipm1}
with $\deg n \leq \ell$ and that,
for a small loop  $\gamma$ around  a perpendicular joint $\frakj$,
the consistency relation $\frakp_{\gamma,\frakD_{\ell}}=\rmid$ holds
 modulo $G^{>\ell}$.
 By Lemma \ref{lem:actZ1}, $\frakp_{\gamma,\frakD_{\ell}}$ acts on $\calS_q(x)$.
 By adding walls around $\frakj$ with wall elements
$\Psi_{a,b}[n]^c$ ($\deg n = \ell +1$)  to  $\frakD_{\ell}$,
we have a scattering diagram $\frakD_{\ell+1}$
 such that $\frakp_{\gamma,\frakD_{\ell+1}}=\rmid$  modulo $G^{>\ell+1}$.
 Since we consider the action of $\Psi_{a,b}[n]^c$ modulo $G^{>\ell+1}$,
 the second expression of  \eqref{eq:xmut1} is simplified as
 \begin{align}
 x^{\tilde m}
 \biggl(1 +
 c 
 \frac{q^{2\langle n,\tilde m \rangle}-1}{q^{2a}-1}q^{a}q^{b} 
 x^{  \tilde p^*(n)}
\biggr).
 \end{align}
 We may further concentrate on $\tilde m$ with  $\langle \d(n)n,\tilde m \rangle=1$,
because other cases are propotional to it with factors that are independent of $\Psi_{a,b}[n]^c$.
In particular, the action of $\Psi_{1/\delta(n),0}[n]^{c}$
on such $\tilde m$ is given by
 \begin{align}
 x^{\tilde m}
 (1 +
 c q^{a}
 x^{  \tilde p^*(n)}
).
 \end{align}
This means that,
for each $n\in N^+$ ($\deg n = \ell +1$),
the (possibly multiple) added wall elements $\Psi_{a,b_i}[n]^{c_i}$ ($i=1$, 2, \dots) in $\frakD_{\ell+1}$ is
replaced with a single wall element
$\Psi_{1/\delta(n),0}[n]^{c}$; moreover, we have $c\in \bbZ[q^{\pm 1/\delta_0}]$ by the induction assumption.
Then, it is  further decomposed
 into a product of
 $\Psi_{1/\delta(n),b}[n]^c$'s  ($c\in \bbZ$).
\end{proof}

\section{Application to   positivity and nonpositivity of QCSDs}
\label{sec:application1}

In the rest of the paper
we apply the pentagon relation,
together with the combinatorics of quantum data,
 to study the positivity and nonpositivity of QCSDs.

\subsection{Positivity of QCSDs}
\label{subsec:positivity1}

Let us briefly summarize known results concerning the {positivity} and nonpositivity of CSDs and QCSDs.

Let us introduce the following notion.
\begin{defn}
\label{defn:positiveCSD1}
A CSD $\frakD_{\fraks}$  is  a \emph{positive realization} if 
 the wall element of any wall has the form
\begin{align}
\label{eq:psipos0}
\Psi[n]^{c\delta(n)}
\quad
(n\in N^+, \ b \in (1/\delta_0)\bbZ,\ c\in \bbZ_{>0}).
\end{align}
We say that a CSD $\frakD_{\fraks}^q$  is \emph{positive}
if it is equivalent to a CSD that is a positive realization.
\end{defn}

The following fact is one of the main results of \cite{Gross14}.

\begin{thm}[{\cite[Theorem 1.13]{Gross14}}]
\label{thm:CSDpos1}
Every CSD $\frakD_{\fraks}$  is positive.
\end{thm}

For each CSD,
a \emph{theta function}  is defined by \emph{broken lines} \cite{Gross14}.
The definition of a broken line and Theorem \ref{thm:CSDpos1} immediately imply
the following positivity result.
  
  \begin{cor}[Positivity of theta functions \cite{Gross14}]
  For any CSD $\frakD_{\fraks}$, every  nonzero coefficients of any theta function is
a positive integer.
  \end{cor}

Based on Proposition \ref{prop:psi1},
we introduced a parallel notion for QCSDs.

\begin{defn}
\label{defn:positive1}
A QCSD $\frakD_{\fraks}^q$  is  a \emph{positive realization} if 
 the wall element of any wall has the form
\begin{align}
\label{eq:psipos1}
\Psi_{1/\d(n),b}[n]^c
\quad
(n\in N^+, \ b \in (1/\delta_0)\bbZ,\ c\in \bbZ_{>0}).
\end{align}
We say that a QCSD $\frakD_{\fraks}^q$  is \emph{positive}
if it is equivalent to a QCSD that is a positive realization.
\end{defn}

\begin{ex}
\label{ex:psipm2}
All examples of QCSDs in Section \ref{sec:ex1} are positive realizations.
\end{ex}

A \emph{quantum theta function} for a QCSD $\frakD_{\fraks}^q$ is defined by \emph{broken lines} \cite{Davison19}
in the same way as in the classical one \cite{Gross14}.
The following fact follows from the definition of a broken line.

\begin{prop}
[{\cite[Theorem 3.9]{Davison19}}]
\label{prop:posif1}
Suppose that $\frakD_{\fraks}^q$ is positive.
Then,
every quantum theta function of a QCSD $\frakD_{\fraks}^q$
has a formal power series of the form
\begin{align}
\label{eq:theta1}
\theta_{Q,\tilde m_0}=x^{\tilde m_0}  \biggl( 1+ \sum_{n\in N^+}
c_n x^{\tilde p^*(n)}
\biggr)
\quad
(c_n\in \bbZ_{\geq 0}[q^{\pm1/\delta_0}]).
\end{align}
\end{prop}

Observing the parallelism between the CSDs and QCSDs so far,
it is natural to expect that Theorem \ref{thm:CSDpos1} also holds for QCSDs.
However, this is not true in general.
We say that a QCSD is \emph{skew-symmetric} if its initial exchange matrix $B$ is skew-symmetric.
The following positivity result is known.
\begin{thm}[{\cite[Theorem 2.15]{Davison19}}]
\label{thm:DM1}
Any skew-symmetric QCSD $\frakD_{\fraks}^q$ is positive.
\end{thm}

The proof of \cite{Davison19} is based on the Donaldson-Thomas theory
on quiver-representations, which applies exclusively to the skew-symmetric case.

On the other hand, the following example suggests the abundance of the nonpositive QCSDs.
\begin{ex}
[\cite{Lee14,Cheung20}]
\label{ex:qgb1}
Consider a cluster algebra of rank 2 with the initial exchange matrix
\begin{align}
\label{eq:bmat1}
B=\begin{pmatrix}
0 & -\d_1\\
\d_2 & 0
\end{pmatrix}
\quad
(\d_1,\d_2\in \bbZ_{>0}).
\end{align}
The (classical) \emph{greedy elements} were introduced in \cite{Lee12}.
Later they were identified with the theta functions  \cite{Cheung15}.
Meanwhile,  the \emph{quantum greedy elements} were studied in \cite{Lee14}.
It turned out that the positivity of the coefficients of the quantum greedy elements \emph{fails}
for $(\d_1,\d_2)\allowbreak=(2,3)$, $(2,5)$, $(3,4)$, $(4,6)$, for example.
Motivated by this result, 
the nonpositivity of \emph{quantum theta functions} for $(\d_1,\d_2)=(2,3)$ was shown by \cite{Cheung20}.
Thus, the corresponding QCSD is \emph{nonpositive} due to Proposition \ref{prop:posif1}.
Also, it was conjectured \cite[Conj.~14]{Lee14} that 
the positivity holds for the quantum greedy basis
if $\d_1 | \d_2$ or $\d_2 | \d_1$.
\end{ex}

Thus, there is an intriguing discrepancy of the positivity
between CSDs and QCSDs especially in the \emph{nonskew-symmetric case}.
Then, we have the following natural questions.
\begin{enumerate}
\item[Q1:]
Exactly when and why does the positivity of a nonskew-symmetric QCSD
hold or fail?
\item[Q2:] When the positivity fails for a QCSD, how 
is the positivity restored in the classical limit $q\rightarrow 1$?
\end{enumerate}

Below we study the above problems for the nonskew-symmetric QCSDs of rank 2
 using the pentagon relation.
 Before that, let us present one immediate and important consequence of the positivity of QCSDs.
 We say that a CSD $\frakD_{\fraks}$ is \emph{with minimal support} if
 $\rmSupp(\frakD_{\fraks})$ is minimal among all CSDs that are equivalent to $\frakD_{\fraks}$.
 If  $\frakD_{\fraks}$ is a positive realization, then $\frakD_{\fraks}$ is with minimal support
 because it is impossible to cancel the wall elements of any walls due to the positivity.
 
 \begin{prop}
 Suppose that a QCSD $\frakD_{\fraks}^q$ is a positive realization.
 Then,  $\rmSupp(\frakD_{\fraks}^q)$ coincides with  $\rmSupp(\frakD_{\fraks})$,
 where $\frakD_{\fraks}$ is a CSD with minimal support.
  \end{prop}
\begin{proof}
By replacing each wall element $\Psi_{1/\delta(n),b}[n]^c$ in $\frakD_{\fraks}^q$ with $\Psi[n]^{c\delta(n)}$,
we obtain a CSD  $\frakD_{\fraks}$,
which is a positive realization.
\end{proof}

\subsection{Rank 2 examples}
\label{subsec:rank21}

Let us focus on the rank 2 case.
As mentioned in Example \ref{ex:psipm2},
all examples with $\d_1\d_2 \leq 4$ in Section \ref{sec:ex1} are positive.
Let us consider examples with $\d_1\d_2 \geq 5$,
namely, QCSDs of nonaffine infinite type.
We use the same conventions and notations
in Section \ref{sec:ex1}.

Let us compute in the lowest approximation, namely, modulo $G^{>2}$.

(a). $(\d_1,\d_2)=(2,3)$.
This is the case where the nonpositivity of a quantum greedy element \cite{Lee14} and (quantum) theta function \cite{Cheung20} was shown.
We have
\begin{align}
\label{eq:23case}
\begin{split}
\begin{bmatrix}
0\\
1
\end{bmatrix}
_{\frac{1}{3},0}
\begin{bmatrix}
1\\
0
\end{bmatrix}
_{\frac{1}{2},0}
& \equiv
\begin{bmatrix}
1\\
0
\end{bmatrix}
_{\frac{1}{2},0}
\begin{bmatrix}
1\\
1
\end{bmatrix}
_{1,-\frac{7}{6}}
\begin{bmatrix}
1\\
1
\end{bmatrix}
_{1,-\frac{1}{2}}
\begin{bmatrix}
1\\
1
\end{bmatrix}
_{1,-\frac{1}{6}}
\begin{bmatrix}
1\\
1
\end{bmatrix}
_{1,\frac{1}{6}}
\begin{bmatrix}
1\\
1
\end{bmatrix}
_{1,\frac{1}{2}}
\begin{bmatrix}
1\\
1
\end{bmatrix}
_{1,\frac{7}{6}}
\begin{bmatrix}
0\\
1
\end{bmatrix}
_{\frac{1}{3},0}
\\
&=
\begin{bmatrix}
1\\
0
\end{bmatrix}
_{\frac{1}{2},0}
\begin{bmatrix}
1\\
1
\end{bmatrix}
_{\frac{1}{6},-\frac{1}{3}}
\begin{bmatrix}
1\\
1
\end{bmatrix}
_{\frac{1}{6},0}^{-1}
\begin{bmatrix}
1\\
1
\end{bmatrix}
_{\frac{1}{6},\frac{1}{3}}
\begin{bmatrix}
0\\
1
\end{bmatrix}
_{\frac{1}{3},0}.
\end{split}
\end{align}
Thus, we confirmed again that it is nonpositive
in the sense of Definition \ref{defn:positive1}.
Moreover, in the classical limit, the negative power in the last expression
 is cancelled
by the positive power,
and we obtain $\Psi[\bfe_1]^2\Psi[(1,1)]^6\Psi[\bfe_2]^3$.
This clarifies the mechanism how the positivity is restored in the classical limit.
Thus, we have a clear answer to Question 2 in Section \ref{subsec:positivity1}.
Moreover,
comparing the  computation in \eqref{eq:23case} with the much more complicated one in \cite[Appendix B]{Cheung20},
we observe that the pentagon relation significantly simplifies the analysis
(just in two lines in \eqref{eq:23case}).

(b). $(\d_1,\d_2)=(2,4)$.
By a similar computation, we have
\begin{align}
\begin{split}
\begin{bmatrix}
0\\
1
\end{bmatrix}
_{\frac{1}{4},0}
\begin{bmatrix}
1\\
0
\end{bmatrix}
_{\frac{1}{2},0}
\equiv
\begin{bmatrix}
1\\
0
\end{bmatrix}
_{\frac{1}{2},0}
\begin{bmatrix}
1\\
1
\end{bmatrix}
_{\frac{1}{4},-\frac{1}{2}}
\begin{bmatrix}
1\\
1
\end{bmatrix}
_{\frac{1}{4},\frac{1}{2}}
\begin{bmatrix}
0\\
1
\end{bmatrix}
_{\frac{1}{4},0}.
\end{split}
\end{align}
This is positive up to $G^{\leq 2}$.

(c). $(\d_1,\d_2)=(3,3)$.
This is skew-symmetric, thus,  positive
by Theorem \ref{thm:DM1}.
Indeed, we have
\begin{align}
\begin{split}
\begin{bmatrix}
0\\
1
\end{bmatrix}
_{\frac{1}{3},0}
\begin{bmatrix}
1\\
0
\end{bmatrix}
_{\frac{1}{3},0}
\equiv
\begin{bmatrix}
1\\
0
\end{bmatrix}
_{\frac{1}{3},0}
\begin{bmatrix}
1\\
1
\end{bmatrix}
_{\frac{1}{3},-\frac{2}{3}}
\begin{bmatrix}
1\\
1
\end{bmatrix}
_{\frac{1}{3},0}
\begin{bmatrix}
1\\
1
\end{bmatrix}
_{\frac{1}{3},\frac{2}{3}}
\begin{bmatrix}
0\\
1
\end{bmatrix}
_{\frac{1}{3},0}.
\end{split}
\end{align}
This is positive up to $G^{\leq 2}$ as it should be.

(d). $(\d_1,\d_2)=(3,4)$.
This is similar to  Example (a).
We have
\begin{align}
\begin{split}
\begin{bmatrix}
0\\
1
\end{bmatrix}
_{\frac{1}{3},0}
\begin{bmatrix}
1\\
0
\end{bmatrix}
_{\frac{1}{4},0}
\equiv
\begin{bmatrix}
1\\
0
\end{bmatrix}
_{\frac{1}{4},0}
\begin{bmatrix}
1\\
1
\end{bmatrix}
_{\frac{1}{12},-\frac{1}{2}}
\begin{bmatrix}
1\\
1
\end{bmatrix}
_{\frac{1}{12},-\frac{1}{3}}
^{-1}
\begin{bmatrix}
1\\
1
\end{bmatrix}
_{\frac{1}{12},0}
\begin{bmatrix}
1\\
1
\end{bmatrix}
_{\frac{1}{12},\frac{1}{3}}
^{-1}
\begin{bmatrix}
1\\
1
\end{bmatrix}
_{\frac{1}{12},\frac{1}{2}}
\begin{bmatrix}
0\\
1
\end{bmatrix}
_{\frac{1}{3},0}.
\end{split}
\end{align}
This is nonpositive.

\subsection{Nonpositivity result}
\label{subsec:nonpositivity1}
Consider a cluster algebra
of rank 2
such that its exchange matrix $B$
has the form in \eqref{eq:bmat1}.
As stated in Example \ref{ex:qgb1},
it was conjectured in  \cite[Conj.~14]{Lee14} that  the positivity
of the quantum greedy basis 
holds if $\d_1|\d_2$ or $\d_2|\d_1$.
The following result is motivated by the converse of this conjecture.

\begin{thm} 
\label{thm:nonpos1}
Let $\frakD^q_{\fraks}$ be a QCSD of rank 2
such that $b_{12}\neq 0$ for $B$ in \eqref{eq:B1}.
Then, $\frakD^q_{\fraks}$ is nonpositive if
$\d_1 \nmid \d_2$ and $\d_2 \nmid \d_1$.
\end{thm}
\begin{proof}
Let us show that  under the condition the nonpositivity already emerges at degree 2
just as above Examples (a) and (d).
We continue to assume that  $\{e_2,e_1\}=1$.
By a similar consideration to the examples in Section \ref{subsec:rank21},
we have, modulo $G^{>2}$,
\begin{gather}
\label{eq:ee1}
\begin{split}
\begin{bmatrix}
0\\
1
\end{bmatrix}
_{\frac{1}{\d_2},0}
\begin{bmatrix}
1\\
0
\end{bmatrix}
_{\frac{1}{\d_1},0}
&=
\Biggl(
\prod_{j=1}^{\d_2}
\begin{bmatrix}
0\\
1
\end{bmatrix}
_{1,\frac{\d_2+1-2j}{\d_2}}
\Biggr)
\Biggl(
\prod_{i=1}^{\d_1}
\begin{bmatrix}
1\\
0
\end{bmatrix}
_{1,\frac{\d_1+1-2i}{\d_1}}
\Biggr)
\\
&
\equiv
\begin{bmatrix}
1\\
0
\end{bmatrix}
_{\frac{1}{\d_1},0}
\Biggl(
\prod_{i=1}^{\d_1}
\prod_{j=1}^{\d_2}
\begin{bmatrix}
1\\
1
\end{bmatrix}
_{1,b(i,j)}
\Biggr)
\begin{bmatrix}
0\\
1
\end{bmatrix}
_{\frac{1}{\d_2},0},
\end{split}
\\
b(i,j)=(2\d_1\d_2 + (1-2j)\d_1 +(1-2i) \d_2)/\d_1\d_2.
\end{gather}
Let $c=\mathrm{gcd}(\d_1,\d_2)$, and let
$\d_1=\d'_1 c$, $\d_2=\d'_2 c$.
Then, $\d((1,1))=\d'_1\d'_2c$,
and we have
\begin{align}
b(i,j)=(2\d'_1\d'_2c + \d'_1 +\d'_2 - 2\d'_1 j - 2\d'_2 i)/\d((1,1)).
\end{align}
Let $m_0$ be the largest number of  the numerator of $b(i,j)$, which is attained with $i=j=1$.
The exponents of $q$ in $\Psi_{1/\d((1,1)),b}[(1,1)]$
align with the interval $2/\d((1,1))$.
Thus,
if the positivity holds, 
we should also have the number  $m_0-2$
 in the numerator.
 This occurs only when $\d'_1=1$ or $\d'_2=1$.
This means $\d_1\mid \d_2$ or $\d_2 \mid \d_1$.
\end{proof}

The above result implies the nonpositivity of (quantum) theta functions.

\begin{prop}
\label{prop:pospos1}
Under the same condition of Theorem \ref{thm:nonpos1},
there is a (quantum) theta function such that at least one of its coefficients is negative in  $\bbZ[q^{\pm1/\delta_0}]$.
\end{prop}

\begin{proof}
Consider the theta function $\theta_{Q,\tilde m_0}$,
where $Q$ is in the first quadrant and  $m_0=(\delta_1,-\delta_2-1)$.
Then, there is a wall $\bfw$ in $\frakD_{\fraks}^s$ whose wall element
is a negative power of $\Psi_{1/\d((1,1)),b}[(1,1)]$.
Then,
the broken line $\gamma$ bending only at $\bfw$  contributing to $\theta_{Q,\tilde m_0}$
 with a negative coefficient in $\bbZ[q^{\pm1/\delta_0}]$.
See Figure \ref{fig:bl1}.
Moreover,  the contribution of $\gamma$ is not canceled by any other broken lines.
Indeed, such broken lines should bend exactly once both at $\bfe_1^{\perp}$ and $\bfe_2^{\perp}$,
but this is impossible as seen in the figure.
\end{proof}

 \begin{figure}
\begin{tikzpicture}[scale=2]
\draw  (-1,0) -- (1,0);
\draw (0,1) -- (0,-1);
\draw (0,0) -- (0.75,-1);
\draw [thick] (0.66,-1) -- (0.3,-0.4);
\draw [thick] (0.3,0.3) -- (0.3,-0.4);
\draw [dashed] (0.48,-1) -- (0,-0.2);
\draw [dashed] (-0.6,0) -- (0,-0.2);
\draw [dashed] (-0.6,0) -- (-0.6,0.6);
\draw  (0.5,0.3)  node {$Q$};
\draw  (0.66,-1.1)  node {$\gamma$};
\draw [fill] (0.3,0.3) circle [radius=0.03];
\end{tikzpicture}
\caption{Broken line for $\theta_{Q,\tilde m_0}$ in Proposition \ref{prop:pospos1} for $(\delta_1,\delta_2)=(3,4)$.
The dashed broken line fails to reach $Q$.}
\label{fig:bl1}
\end{figure}

We have the following corollary of Theorem \ref{thm:nonpos1}.
\begin{cor}
\label{cor:nonpos2}
Let $\frakD^q_{\fraks}$ be a QCSD of any rank,
and let $B$ be the initial exchange matrix in \eqref{eq:B1}.
Then, $\frakD^q_{\fraks}$ is nonpositive
if there is a pair $i\neq j$ with $b_{ij}\neq 0$ such that 
 $\d_i\nmid  \d_j$ and  $\d_j\nmid \d_i$.
 \end{cor}
\begin{proof}
Suppose that
 $\d_i \nmid \d_j$ holds for a pair  $i\neq j$ with $b_{ij}\neq 0$.
 We consider a QCSD whose wall elements are
 in the form \eqref{eq:psipm1}.
 Then,
the proof of Theorem \ref{thm:nonpos1}
tells that there is some wall whose wall element is a negative power of $\Psi_{1/\d(e_i+e_j),b}[e_i+e_j]$.
Moreover, due to the construction of a CSD in \cite[Appendix C.1]{Gross14},
other wall elements never cancel the above wall element.
\end{proof}

\subsection{Positivity up to   degree 4}
\label{subsec:positivity2}
We are left with the problem to determine
whether the converses of Theorem \ref{thm:nonpos1} and 
Corollary
 \ref{cor:nonpos2} hold or not.
 Here, we concentrate on the rank 2 case
 and examine the positivity up to degree 4.

 First, let us give some examples.
\begin{ex}
\label{ex:further1}
(a).
Let  $(\d_1,\d_2)=(1,5)$.
By a direct computation with the pentagon relation,
we verify the following relation modulo $G^{>6}$:
\begin{align*}
\begin{split}
\begin{bmatrix}
0\\
1
\end{bmatrix}
_{\frac{1}{5},0}
\begin{bmatrix}
1\\
0
\end{bmatrix}
_{1,0}
&
\equiv
\begin{bmatrix}
1\\
0
\end{bmatrix}
_{1,0}
\begin{bmatrix}
1\\
1
\end{bmatrix}
_{\frac{1}{5},0}
\begin{bmatrix}
2\\
3
\end{bmatrix}
_{\frac{1}{5},-\frac{2}{5}}
\begin{bmatrix}
2\\
3
\end{bmatrix}
_{\frac{1}{5},\frac{2}{5}}
\begin{bmatrix}
1\\
2
\end{bmatrix}
_{\frac{1}{5},-\frac{2}{5}}
\begin{bmatrix}
1\\
2
\end{bmatrix}
_{\frac{1}{5},\frac{2}{5}}
\\
&
\qquad
\times
\begin{bmatrix}
2\\
4
\end{bmatrix}
_{\frac{2}{5},-\frac{6}{5}}
\begin{bmatrix}
2\\
4
\end{bmatrix}
_{\frac{2}{5},-\frac{4}{5}}
\begin{bmatrix}
2\\
4
\end{bmatrix}
_{\frac{2}{5},\frac{4}{5}}
\begin{bmatrix}
2\\
4
\end{bmatrix}
_{\frac{2}{5},\frac{6}{5}}
\\
&
\qquad
\times
\begin{bmatrix}
1\\
3
\end{bmatrix}
_{\frac{1}{5},-\frac{2}{5}}
\begin{bmatrix}
1\\
3
\end{bmatrix}
_{\frac{1}{5},\frac{2}{5}}
\begin{bmatrix}
1\\
4
\end{bmatrix}
_{\frac{1}{5},0}
\begin{bmatrix}
1\\
5
\end{bmatrix}
_{1,0}
\begin{bmatrix}
0\\
1
\end{bmatrix}
_{\frac{1}{5},0}.
\end{split}
\end{align*}
Thus, it is positive up to $G^{\leq 6}$.
\par
(b).
Let  $(\d_1,\d_2)=(2,4)$.
We have the following relation modulo $G^{>4}$:
\begin{align*}
\begin{split}
\begin{bmatrix}
0\\
1
\end{bmatrix}
_{\frac{1}{4},0}
\begin{bmatrix}
1\\
0
\end{bmatrix}
_{\frac{1}{2},0}
&
\equiv
\begin{bmatrix}
1\\
0
\end{bmatrix}
_{\frac{1}{2},0}
\begin{bmatrix}
2\\
1
\end{bmatrix}
_{\frac{1}{4},0}
\begin{bmatrix}
1\\
1
\end{bmatrix}
_{\frac{1}{4},-\frac{1}{2}}
\begin{bmatrix}
1\\
1
\end{bmatrix}
_{\frac{1}{4},\frac{1}{2}}
\begin{bmatrix}
2\\
2
\end{bmatrix}
_{\frac{1}{2},-\frac{3}{2}}
\begin{bmatrix}
2\\
2
\end{bmatrix}
_{\frac{1}{2},-1}
\\
&\qquad  \times
\begin{bmatrix}
2\\
2
\end{bmatrix}
_{\frac{1}{2},-\frac{1}{2}}
\begin{bmatrix}
2\\
2
\end{bmatrix}
_{\frac{1}{2},\frac{1}{2}}
\begin{bmatrix}
2\\
2
\end{bmatrix}
_{\frac{1}{2},1}
\begin{bmatrix}
2\\
2
\end{bmatrix}
_{\frac{1}{2},\frac{3}{2}}
\begin{bmatrix}
1\\
2
\end{bmatrix}
_{\frac{1}{2},-1}
\begin{bmatrix}
1\\
2
\end{bmatrix}
_{\frac{1}{2},-\frac{1}{2}}
\\
&\qquad  \times
\begin{bmatrix}
1\\
2
\end{bmatrix}
_{\frac{1}{2},0}
\begin{bmatrix}
1\\
2
\end{bmatrix}
_{\frac{1}{2},0}
\begin{bmatrix}
1\\
2
\end{bmatrix}
_{\frac{1}{2},\frac{1}{2}}
\begin{bmatrix}
1\\
2
\end{bmatrix}
_{\frac{1}{2},1}
\begin{bmatrix}
1\\
3
\end{bmatrix}
_{\frac{1}{4},-\frac{1}{2}}
\begin{bmatrix}
1\\
3
\end{bmatrix}
_{\frac{1}{4},\frac{1}{2}}
\begin{bmatrix}
0\\
1
\end{bmatrix}
_{\frac{1}{4},0}.
\end{split}
\end{align*}
Thus, it is positive up to $G^{\leq 4}$.
\end{ex}

We can do the same computation 
up to $G^{\leq 4}$ in full generality
and confirm the positivity as below.

\begin{ex}
\label{ex:upto4}
Suppose that $\d_2=k \d_1$ for some positive integer $k$.
Applying the pentagon relation and the fission/fusion formula
as before,
we  obtain the following explicit formulas of the factors in the ordered product
for the LHS
of \eqref{eq:ee1} up to degree 4.

\par
(a).
{\em Degree 2}.
We have
$\d((1,1))=k\d_1$, and
the formula in \eqref{eq:ee1} is specialized as
\begin{gather}
\label{eq:ee2}
\prod_{j=1}^{k\d_1}
\prod_{i=1}^{\d_1}
\begin{bmatrix}
1\\
1
\end{bmatrix}
_{1,\frac{2k\d_1 + 1 + k - 2 j - 2k i}{k\d_1}}
=
\prod_{i=1}^{\d_1}
\begin{bmatrix}
1\\
1
\end{bmatrix}
_{\frac{1}{k\d_1},\frac{\d_1+1-2i}{\d_1}}.
\end{gather}

\par
(b).
{\em Degree 3}.
We have
\begin{align}
\label{eq:deg32}
\d((2,1))=k\d_1,
\quad
\d((1,2))
=
\begin{cases}
k\d_1/2
& (k: \text{even}),
\\
k\d_1
&
(k: \text{odd}).
\end{cases}
\end{align}
Then, we have the terms for $n=(2,1)$ and $(1,2)$,
\begin{gather}
\label{eq:deg30}
\prod_{j=1}^{k\d_1}
\prod_{i=1}^{\d_1-1}
\prod_{t=1}^i 
\begin{bmatrix}
2\\
1
\end{bmatrix}
_{1,\frac{-k\d_1 + 1  - 2 j + 2k i+2kt}{k\d_1}}
=
\prod_{i=1}^{\d_1-1}
\prod_{t=1}^{i}
\begin{bmatrix}
2\\
1
\end{bmatrix}
_{\frac{1}{k\d_1},\frac{-2\d_1+2i+2t}{\d_1}},
\\
\label{eq:deg31}
\begin{split}
\prod_{j=1}^{k\d_1-1}
\prod_{i=1}^{\d_1}
\prod_{t=1}^{k\d_1-j}
\begin{bmatrix}
1\\
2
\end{bmatrix}
_{1,\frac{k\d_1 + k  - 2 j - 2k i+2t}{k\d_1}}
\hskip100pt
\\
=
\begin{cases}
\displaystyle
\prod_{i=1}^{\d_1}
\prod_{t=1}^{k\d_1-1}
\begin{bmatrix}
1\\
2
\end{bmatrix}
_{\frac{2}{k\d_1},\frac{k-2ki+2t}{k\d_1}}
&
(k\d_1: \text{even}),
\\
\displaystyle
\prod_{i=1}^{\d_1}
\prod_{t=1}^{(k\d_1-1)/2}
\begin{bmatrix}
1\\
2
\end{bmatrix}
_{\frac{1}{k\d_1},\frac{k-1-2ki+4t}{k\d_1}}
&
(k\d_1: \text{odd}).
\end{cases}
\end{split}
\end{gather}
The equality \eqref{eq:deg30} is straightforward,
while the one \eqref{eq:deg31}
follows from the following
resummation
formula for any positive integer $p$:
\begin{align}
\begin{split}
\sum_{j=1}^{p-1}
\sum_{t=1}^{p-j}
x^{-2j+2t}
&=
\frac{x^{-2p+4}(1-x^{2p-2})(1-x^{2p})}
{(1-x^2)(1-x^4)}
\\
&
=\begin{cases}
\displaystyle
\sum_{j=1}^{p/2}
\sum_{t=1}^{p-1}
x^{2-4j+2t}
&
(p: \text{even}),
\\
\displaystyle
\sum_{j=1}^{p}
\sum_{t=1}^{(p-1)/2}
x^{-2j+4t}
&
(p: \text{odd}).
\end{cases}
\end{split}
\end{align}
To see the positivity,
we  need to resolve the discrepancy between 
\eqref{eq:deg32} and \eqref{eq:deg31} 
for $n=(1,2)$ when $\d_1$ is even and $k$ is odd,
e.g., $(\d_1,\d_2)=(2,6)$.
In fact, in this case we can further rewrite
the RHS of  \eqref{eq:deg31} (for even $k\d_1$) as
\begin{align}
\label{eq:deg33}
\begin{split}
&
\prod_{i=1}^{\d_1/2}
\Biggl\{
\Biggl(
\prod_{t=1}^{(k+1)/2}
\begin{bmatrix}
1\\
2
\end{bmatrix}
_{\frac{1}{k\d_1},\frac{k-1-4ki+4t}{k\d_1}}
\begin{bmatrix}
1\\
2
\end{bmatrix}
_{\frac{1}{k\d_1},\frac{k+2k\d_1-5-4ki+4t}{k\d_1}}
\Biggr)
\prod_{t=1}^{k\d_1-k-2}
\begin{bmatrix}
1\\
2
\end{bmatrix}
_{\frac{1}{k\d_1},\frac{3k+1-4ki+2t}{k\d_1}}
\Biggr\}.
\end{split}
\end{align}

\par
(c).
{\em Degree 4}.
The results are similar to the degree 3 case.
However, the formulas become lengthy, so we put them in Appendix \ref{sec:positivity}.
They also demonstrate the complexity of proving the positivity in higher degrees by direct computation.
\end{ex}

We summarize the above results in Example \ref{ex:upto4}
as follows.
\begin{prop}
The converse of Theorem \ref{thm:nonpos1}
holds at least up to 
$G^{\leq 4}$.
\end{prop}

Encouraged by this partial but already nontrivial  result,
we give a conjecture in the spirit of \cite[Conj.~14]{Lee14}.

\begin{conj}[{Cf.~\cite[Conj.~14]{Lee14}}]
 The converse of  
Theorem \ref{thm:nonpos1} holds.
\end{conj}

\appendix

\section{Derivation of relations \eqref{eq:a115} and \eqref{eq:a227} by pentagon relation
}

\label{sec:derivation}
We faithfully follow the derivation of \cite{Matsushita21}
in the classical case.
We ask  the reader to consult \cite{Matsushita21}
for further details.
It is important that in the derivation we only use the pentagon and commutative relations in Theorem \ref{thm:pent1}.

\subsection{Derivation of relation \eqref{eq:a115}}

Let us write $\Psi_{a,b}[n]$ as $[n]_{a,b}$ for simplicity.

\begin{lem}[{cf.~\cite[Lemma 1]{Matsushita21}}]
\label{lem:a111}
If $\{n',n\}=c>0$, the following formulas hold:
\begin{align}
\label{eq:nn1}
[n']_{c/2,b'}[n]_{c,b}
&=[n]_{c,b}
[n+n']_{c/2,b+b'}
[n+2n']_{c,b+2b'}
[n']_{c/2,b'},
\\
[n']_{c,b'}[n]_{c/2,b}
&=[n]_{c/2,b}
[2n+n']_{c,2b+b'}
[n+n']_{c/2,b+b'}
[n']_{c,b'}.
\end{align}
\end{lem}
\begin{proof}
This is proved by the fission/fusion and the pentagon relation
in the same way as the relation \eqref{eq:pent3}.
\end{proof}

\begin{lem}[{cf.~\cite[Lemma 2]{Matsushita21}}]
\label{lem:a112}
Let $\ell$ be any nonnegative integer.
If $\{n',n\}=c>0$, the following formula holds:
\begin{align}
\begin{split}
&\
[n']_{c/2,b'}
\biggl(\,
\buildrel \longrightarrow \over {\prod_{0\leq p \leq \ell}}
[n+2pn']_{c,b+2pb'\pm pc}
\biggr)
\\
=&\
[n]_{c,b}
\biggl(\,
\buildrel \longrightarrow \over {\prod_{1\leq p \leq 2\ell+1}}
[n+pn']_{c/2,b+pb'\pm (p-1)c/2}
\biggr)
[n+(2\ell+2)n']_{c,b+(2\ell+2)b'\pm\ell c}[n']_{c/2,b'},
\end{split}
\\
\begin{split}
&\
\biggl(\,
\buildrel \longleftarrow \over {\prod_{0\leq p \leq \ell}}
[2pn+n']_{c,2pb+b'\pm pc}
\biggr)
[n]_{c/2,b}
\\
=&\
[n]_{c/2,b}
[(2\ell+2)n+n']_{c,(2\ell+2)b+b'\pm\ell c}
\biggl(\,
\buildrel \longleftarrow \over {\prod_{1\leq p \leq 2\ell+1}}
[pn+n']_{c/2,pb+b'\pm (p-1)c/2}
\biggr)
[n']_{c,b'}.
\end{split}
\end{align}
\end{lem}
\begin{proof}
For $\ell=0$, they coincide with the relations in Lemma \ref{lem:a111}.
Then, they are proved by the induction on $\ell$
with the pentagon relation and Lemma \ref{lem:a111}
in the same way as \cite[Lemma 2]{Matsushita21}.
\end{proof}

\begin{lem}[{cf.~\cite[Lemma 3]{Matsushita21}}]
\label{lem:a113}
Let $\ell$ be any nonnegative integer.
If $\{n',n\}=c>0$, the following formulas hold:
\begin{align}
\begin{split}
[n']_{c/2,b'}
\biggl(\,
\buildrel \longrightarrow \over {\prod_{p\geq 0}}
[n+2pn']_{c,b+2pb'\pm pc}
\biggr)
&=
[n]_{c,b}
\biggl(\,
\buildrel \longrightarrow \over {\prod_{p\geq 1}}
[n+pn']_{c/2,b+pb'\pm (p-1)c/2}
\biggr)
[n']_{c/2,b'},
\end{split}
\\
\begin{split}
\biggl(\,
\buildrel \longleftarrow \over {\prod_{p\geq 0}}
[2pn+n']_{c,2pb+b'\pm pc}
\biggr)
[n]_{c/2,b}
&=
[n]_{c/2,b}
\biggl(\,
\buildrel \longleftarrow \over {\prod_{p\geq 1}}
[pn+n']_{c/2,pb+b'\pm (p-1)c/2}
\biggr)
[n']_{c,b'}.
\end{split}
\end{align}
\end{lem}
\begin{proof}
This is obtained by taking the limit $\ell\rightarrow \infty$
of the relations in Lemma \ref{lem:a112}.
\end{proof}

\begin{thm}[{cf.~\cite[Theorem 2]{Matsushita21}}]
\label{thm:a111}
If $\{n',n\}=c>0$, the following formula holds:
\begin{align}
\label{eq:nn3}
\begin{split}
&\ [n']_{c/2,b'}[n]_{c/2,b}
\\
=&\
\biggl(\,
\buildrel \longrightarrow \over {\prod_{p\geq 0}}
[(p+1)n+pn')]_{c/2,(p+1)b+pb'}
\biggr)
\\
&\quad\times
\biggl(\,
\prod_{p\geq 0}
[2^p(n+n')]_{2^{p-1}c,2^p(b+b')-2^{p-1}c}
[2^p(n+n')]_{2^{p-1}c,2^p(b+b')+2^{p-1}c}
\biggr)
\\
&\quad\times
\biggl(\,
\buildrel \longleftarrow \over {\prod_{p\geq 0}}
[pn+(p+1)n')]_{c/2,pb+(p+1)b'}
\biggr).
\end{split}
\end{align}
In particular, if we set $n=\bfe_1$, $n'=\bfe_2$, $c=1$, and $b=b'=0$, we obtain
the relation \eqref{eq:a115}.
\end{thm}
\begin{proof}
It is enough to prove 
that, for a given positive integer $k$, 
the relation holds modulo $G^{>\ell}$ with $\ell=k\deg(n+n')-1$.
We prove it by the induction on $k$.
For $k=1$,
the RHS reduced to
\begin{align}
[n]_{c/2,b}
[n']_{c/2,b'}
\mod G^{>\ell}.
\end{align}
Then, the relation certainly holds  because
$[n]_{c/2,b}=[n]_{c,b-c/2}[n]_{c,b+c/2}$ and $[n']_{c/2,b'}=[n']_{c,b'-c/2}[n']_{c,b'+c/2}$ commute modulo  $G^{>\ell}$
by the pentagon relation.
Suppose that the claim holds for $k-1$ for some $k\geq 2$.
First, we consider the relation
\begin{align}
\label{eq:prod1}
\begin{split}
&\
[n']_{c/2,b'}[n]_{c/2,b}
\\
=&\
( [n']_{c/2,b'}[n]_{c,b-c/2} )[n]_{c,b+c/2}
\\
=&\
[n]_{c,b-c/2}
[n+n']_{c/2,b+b'-c/2}
[n+2n']_{c,b+2b'-c/2}
([n']_{c/2,b'}[n]_{c,b+c/2})
\\
=&\
[n]_{c,b-c/2}
[n+n']_{c/2,b+b'-c/2}
[n+2n']_{c,b+2b'-c/2}
\\
&
\quad
\times
[n]_{c,b+c/2}
[n+n']_{c/2,b+b'+c/2}
[n+2n']_{c,b+2b'+c/2}
[n']_{c/2,b'},
\end{split}
\end{align}
where we used \eqref{eq:nn1} twice for the parentheses.
This calculation is essentially the same as \eqref{eq:com2}.
We need to calculate the anti-ordered product
$[n+2n']_{c,b+2b'-c/2}
[n]_{c,b+c/2}
$
in the last expression.
Note that $\{n+2n',n\}=2c$.
Then, by the induction hypothesis, we have, modulo $G^{\ell'}$
with $\ell'=(k-1)\deg(n+(n+2n'))$,
\begin{align}
\label{eq:prod2}
\begin{split}
&\ [n+2n']_{c,b+2b'-c/2}[n]_{c,b+c/2}
\\
\equiv &\
\biggl(\,
\buildrel \longrightarrow \over {\prod_{p\geq 0}}
[(2p+1)n+2pn')]_{c,(2p+1)b+pb'+c/2}
\biggr)
\\
&\quad\times
\biggl(\,
\prod_{p\geq 0}
[2^{p+1}(n+n')]_{2^{p}c,2^{p+1}(b+b')-2^{p}c}
[2^{p+1}(n+n')]_{2^{p}c,2^{p+1}(b+b')+2^{p}c}
\biggr)
\\
&\quad\times
\biggl(\,
\buildrel \longleftarrow \over {\prod_{p\geq 0}}
[(2p+1)n+(2p+2)n']_{c,(2p+1)b+(2p+2)b'-c/2}
\biggr).
\end{split}
\end{align}
Meanwhile,
we have the inequality
\begin{align}
(k-1)\deg(n+(n+2n'))=2(k-1)\deg(n+n')\geq k \deg(n+n'),
\end{align}
where the second inequality holds due to the assumption $k\geq 2$.
Therefore, the relation holds also modulo $G^{>\ell}$ with
$\ell= k \deg(n+n')$.
We put \eqref{eq:prod2} into the last expression of \eqref{eq:prod1}
and apply Lemma \ref{lem:a113} to terms therein as
\begin{align}
\begin{split}
&\
[n+n']_{c/2,b+b'-c/2}
\biggl(\,
\buildrel \longrightarrow \over {\prod_{p\geq 0}}
[(2p+1)n+2pn']_{c,(2p+1)b+2pb'+c/2}
\biggr)
\\
=&\
[n]_{c,b+c/2}
\biggl(\,
\buildrel \longrightarrow \over {\prod_{p\geq 1}}
[(p+1)n+pn']_{c/2,(p+1)b+pb'}
\biggr)
[n+n']_{c/2,b+b'-c/2},
\end{split}
\end{align}
\begin{align}
\begin{split}
&\
\biggl(\,
\buildrel \longleftarrow \over {\prod_{p\geq 0}}
[(2p+1)n+(2p+2)n']_{c,(2p+1)b+(2p+2)b'-c/2}
\biggr)
[n+n']_{c/2,b+b'+c/2}
\\
=&\
[n+n']_{c/2,b+b'+c/2}
\biggl(\,
\buildrel \longleftarrow \over {\prod_{p\geq 1}}
[(p+1)n+(p+2)n']_{c/2,(p+1)b+(p+2)b'}
\biggr)
[n+2n']_{c,b+2b'-c/2}.
\end{split}
\end{align}
Then, after a little manipulation, we obtain the desired expression
in the RHS of \eqref{eq:nn1}.
\end{proof}

\subsection{Derivation of relation \eqref{eq:a227}}

\begin{lem}[{cf.~\cite[Lemma 4]{Matsushita21}}]
\label{lem:a222}
Let $\ell$ be any nonnegative integer.
If $\{n',n\}=c>0$, the following formula holds:
\begin{align}
\begin{split}
&\
[n']_{c,b'}
\biggl(\,
\buildrel \longrightarrow \over {\prod_{0\leq p \leq \ell}}
[n+pn']_{c/2,b+pb'\pm pc/2}
\biggr)
\\
=&\
[n]_{c/2,b}
\biggl(\,
\buildrel \longrightarrow \over {\prod_{1\leq p \leq \ell}}
[2n+(2p-1)n']_{c,2b+(2p-1)b'\pm (p-1)c}
[n+pn']_{c/4,b+pb'\pm(2p-1)/4c}
\biggr)
\\
&\quad \times
[2n+(2\ell+1)n']_{c,2b+(2\ell+1)b'\pm\ell c}
[n+(\ell+1)n']_{c/2,b+(\ell+1)b'\pm\ell c/2}
[n']_{c,b'},
\end{split}
\end{align}
\begin{align}
\begin{split}
&\
\biggl(\,
\buildrel \longleftarrow \over {\prod_{0\leq p \leq \ell}}
[pn+n']_{c/2,pb+b'\pm pc/2}
\biggr)
[n]_{c,b}
\\
=&\
[n]_{c,b}
[(\ell+1)n+n']_{c/2,(\ell+1)b+b'\pm\ell c/2}
[(2\ell+1)n+2n']_{c,(2\ell+1)b+2b'\pm\ell c}
\\
&\quad\times
\biggl(\,
\buildrel \longleftarrow \over {\prod_{1\leq p \leq \ell}}
[pn+n']_{c/4,pb+b'\pm (2p-1)c/4}
[(2p-1)n+2n']_{c,(2p-1)b+2b'\pm (p-1)c}
\biggr)
[n']_{c/2,b'}.
\end{split}
\end{align}
\end{lem}
\begin{proof}
For $\ell=0$, they coincide with the relations in Lemma \ref{lem:a111}.
Then, they are proved by the induction on $\ell$
with the pentagon relation and Lemma \ref{lem:a111}
in the same way as \cite[Lemma 4]{Matsushita21}.
\end{proof}

\begin{lem}[{cf.~\cite[Lemma 3]{Matsushita21}}]
\label{lem:a23}
Let $\ell$ be any nonnegative integer.
If $\{n',n\}=c>0$, the following formulas hold:
\begin{align}
\begin{split}
&\
[n']_{c,b'}
\biggl(\,
\buildrel \longrightarrow \over {\prod_{p\geq 0}}
[n+pn']_{c/2,b+pb'\pm pc/2}
\biggr)
\\
=&\
[n]_{c/2,b}
\biggl(\,
\buildrel \longrightarrow \over {\prod_{p\geq 1}}
[2n+(2p-1)n']_{c,2b+(2p-1)b'\pm (p-1)c}
[n+pn']_{c/4,b+pb'\pm(2p-1)/4c}
\biggr)
[n']_{c,b'},
\end{split}
\end{align}
\begin{align}
\begin{split}
&\
\biggl(\,
\buildrel \longleftarrow \over {\prod_{p\geq 0}}
[pn+n']_{c/2,pb+b'\pm pc/2}
\biggr)
[n]_{c,b}
\\
=&\
[n]_{c,b}
\biggl(\,
\buildrel \longleftarrow \over {\prod_{ p\geq 1}}
[pn+n']_{c/4,pb+b'\pm (2p-1)c/4}
[(2p-1)n+2n']_{c,(2p-1)b+2b'\pm (p-1)c}
\biggr)
[n']_{c/2,b'}.
\end{split}
\end{align}

\end{lem}
\begin{proof}
This is obtained by taking the limit $\ell\rightarrow \infty$
of the relations in Lemma \ref{lem:a222}.
\end{proof}

\begin{thm}[{cf.~\cite[Theorem 3]{Matsushita21}}]
\label{thm:a221}
If $\{n',n\}=c>0$, the following formula holds:
\begin{align}
\label{eq:nn5}
\begin{split}
&\ [n']_{c/4,b'}[n]_{c,b}
\\
=&\
\biggl(\,
\buildrel \longrightarrow \over {\prod_{p\geq 0}}
[(2p+1)n+4p n']_{c,(2p+1)b+4pb'}
[(p+1)n+(2p+1) n']_{c/4,(p+1)b+(2p+1)b'}
\biggr)
\\
&\quad\times
[n+2n']_{c/2,b+2b'}
\\
&\quad\times
\biggl(\,
\prod_{p\geq 0}
[2^p(n+2n')]_{2^{p-1}c,2^p(b+2b')-2^{p-1}c}
[2^p(n+2n')]_{2^{p-1}c,2^p(b+2b')+2^{p-1}c}
\biggr)
\\
&\quad\times
\biggl(\,
\buildrel \longleftarrow \over {\prod_{p\geq 0}}
[(2p+1)n+(4p+4)n']_{c,(2p+1)b+(4p+4)b'}
[pn+(2p+1)n']_{c/4,pb+(2p+1)b'}
\biggr).
\end{split}
\end{align}
In particular, if we set $n=\bfe_1$, $n'=\bfe_2$, $c=1$, and $b=b'=0$, we obtain
the relation \eqref{eq:a115}.
\end{thm}
\begin{proof}
In constrast to the proof of Theorem 
\ref{thm:a111},
we do not use the induction;
rather, we use Theorem 
\ref{thm:a111}.
First, we consider the relation
\begin{align}
\label{eq:prod3}
\begin{split}
&\
[n']_{c/4,b'}[n]_{c,b}
\\
=&\
[n']_{c/2,b'-c/4}( [n']_{c/2,b'+c/4} [n]_{c,b})
\\
=&\
([n']_{c/2,b'-c/4}
[n]_{c,b})
[n+n']_{c/2,b+b'+c/4}
[n+2n']_{c,b+2b'+c/2}
[n']_{c/2,b'+c/4}
\\
=&\
[n]_{c,b}
[n+n']_{c/2,b+b'-c/4}
[n+2n']_{c,b+2b'-c/2}
\\
&\quad
\times
[n']_{c/2,b'-c/4}
[n+n']_{c/2,b+b'+c/4}
[n+2n']_{c,b+2b'+c/2}
[n']_{c/2,b'+c/4},
\end{split}
\end{align}
where we used \eqref{eq:nn1} twice for the parentheses.
We need to calculate the anti-ordered product
$
[n']_{c/2,b'-c/4}
[n+n']_{c/2,b+b'+c/4}
$
in the last expression.
By Theorem \ref{thm:a111}, we have
\begin{align}
\label{eq:prod4}
\begin{split}
&\ 
[n']_{c/2,b'-c/4}
[n+n']_{c/2,b+b'+c/4}
\\
= &\
\biggl(\,
\buildrel \longrightarrow \over {\prod_{p\geq 0}}
[(p+1)n + (2p+1)n']_{c,(p+1)b + (2p+1)b'+c/4}
\biggr)
\\
&\quad\times
\prod_{p\geq 0}
[2^{p}(n+2n')]_{2^{p-1}c,2^{p}(b+2b')-2^{p-1}c}
[2^{p}(n+2n')]_{2^{p-1}c,2^{p}(b+2b')+2^{p-1}c}
\\
&\quad\times
\biggl(\,
\buildrel \longleftarrow \over {\prod_{p\geq 0}}
[pn+(2p+1)n']_{c,pb+(2p+1)b'-c/4}
\biggr).
\end{split}
\end{align}
We put \eqref{eq:prod4} into the last expression of \eqref{eq:prod3}
and apply Lemma \ref{lem:a222} to terms therein as
\begin{align}
\begin{split}
&\
[n+2n']_{c,b+2b'-c/2}
\biggl(\,
\buildrel \longrightarrow \over {\prod_{p\geq 0}}
[(p+1)n + (2p+1)n']_{c,(p+1)b + (2p+1)b'+c/4}
\biggr)
\\
=&\
[n+n']_{c/2,b+b'+c/4}
\\
&\quad \times
\biggl(\,
\buildrel \longrightarrow \over {\prod_{p\geq 1}}
[(2p+1)n+4pn']_{c,(2p+1)b + 4pb'}
[(p+1)n+(2p+1)n']_{c/4,(p+1)b+(2p+1)b'}
\biggr)
\\
&\quad \times
[n+2n']_{c,b+2b'-c/2},
\end{split}
\end{align}
\begin{align}
\begin{split}
&\
\biggl(\,
\buildrel \longleftarrow \over {\prod_{p\geq 0}}
[pn+(2p+1)n']_{c,pb+(2p+1)b'-c/4}
\biggr)
[n+2n']_{c,b+2b'+c/2}
\\
=&\
[n+2n']_{c,b+2b'+c/2}
\\
&\quad\times
\biggl(\,
\buildrel \longleftarrow \over {\prod_{p \geq 1}}
[pn+(2p+1)n']_{c/4,pb+(2p+1)b'}
[(2p-1)n+4pn']_{c,(2p-1)b+4pb'}
\biggr)
\\
&\quad\times
[n']_{c/2,b'-c/4}.
\end{split}
\end{align}
Then, after a little manipulation, we obtain the desired expression
in the RHS of \eqref{eq:nn5}.
\end{proof}

\section{Positivity at degree 4}

\label{sec:positivity}

This is a continuation of Example \ref{ex:upto4}
for degree 4.
The purpose of writing down the explicit expressions is to demonstrate
how the problem becomes complicated  even for
such a small degree. 
We have
\begin{align}
\d((3,1))&=k\d_1,
\\
\d((2,2))
&=
\begin{cases}
k\d_1/2
& (k\d_1: \textrm{even}),
\\
k\d_1
& (k\d_1: \textrm{odd}),
\end{cases}
\\
\label{eq:deg41}
\d((1,3))
&=
\begin{cases}
k\d_1/3
& (k\equiv 0 \mod 3),
\\
k\d_1
&
(k\not\equiv 0 \mod 3).
\end{cases}
\end{align}

Firstly, we have the  term for $n=(3,1)$,
\begin{align}
\begin{split}
\prod_{j=1}^{k\d_1}
\prod_{i=1}^{\d_1-2}
\prod_{t=1}^{i}
\prod_{s=1}^{t}
\begin{bmatrix}
3\\
1
\end{bmatrix}
_{1,\frac{-2k \d_1+3k+1-2j+2ki+2kt+2ks}{k\d_1}}
\\
\qquad
=
\prod_{i=1}^{\d_1-2}
\prod_{t=1}^{i}
\prod_{s=1}^{t}
\begin{bmatrix}
3\\
1
\end{bmatrix}
_{\frac{1}{k\d_1},\frac{-3\d_1+3+2i+2t+2s}{\d_1}}.
\end{split}
\end{align}

Next, we have the term for $n=(2,2)$,
\begin{align}
\label{eq:deg44}
\prod_{j=1}^{k\d_1}
\prod_{s=1}^{j-1}
\prod_{i=1}^{\d_1-1}
\prod_{t=1}^{i}
\begin{bmatrix}
2\\
2
\end{bmatrix}
_{1,\frac{-2j-2s+2ki+2kt+2}{k\d_1}-1}
\begin{bmatrix}
2\\
2
\end{bmatrix}
_{1,\frac{-2j-2s+2ki+2kt+2}{k\d_1}+1},
\end{align}
which is written as
\begin{gather}
\begin{split}
&
\prod_{s=1}^{k\d_1-1}
\prod_{i=1}^{\d_1-1}
\prod_{t=1}^{i}
\begin{bmatrix}
2\\
2
\end{bmatrix}
_{\frac{2}{k\d_1},\frac{-3k\d_1 +2s + 2ki+2kt}{k\d_1}-1}
\begin{bmatrix}
2\\
2
\end{bmatrix}
_{\frac{2}{k\d_1},\frac{-3k\d_1  +2s + 2ki+2kt}{k\d_1}+1}
\\
&
\hskip220pt
(k\d_1: \text{even}),
\end{split}
\\
\begin{split}
&\prod_{s=1}^{(k\d_1-1)/2}
\prod_{i=1}^{\d_1-1}
\prod_{t=1}^{i}
\begin{bmatrix}
2\\
2
\end{bmatrix}
_{\frac{1}{k\d_1},\frac{-3k\d_1 -1 +4s + 2ki+2kt}{k\d_1}-1}
\begin{bmatrix}
2\\
2
\end{bmatrix}
_{\frac{1}{k\d_1},\frac{-3k\d_1 -1 +4s + 2ki+2kt}{k\d_1}+1}
\\
&\hskip250pt
(k\d_1: \text{odd}).
\end{split}
\end{gather}
Here we used
 the following resummation formula
for any positive integer $p$:
\begin{align}
\begin{split}
\sum_{j=1}^{p}
\sum_{s=1}^{j-1}
x^{-2j-2s}
&=
\frac{x^{-4p+2}(1-x^{2p-2})(1-x^{2p})}
{(1-x^2)(1-x^4)}
\\
&
=\begin{cases}
\displaystyle
\sum_{j=1}^{p/2}
\sum_{s=1}^{p-1}
x^{-2p-4j+2s}
&
(p: \text{even}),
\\
\displaystyle
\sum_{j=1}^{p}
\sum_{s=1}^{(p-1)/2}
x^{-2p-2-2j+4s}
&
(p: \text{odd}).
\end{cases}
\end{split}
\end{align}

Finally, we have the term for $n=(1,3)$,
\begin{gather}
\begin{split}
\prod_{j=1}^{k\d_1-2}
\prod_{i=1}^{\d_1}
\prod_{t=1}^{k\d_1-j-1}
\prod_{s=1}^{t}
\begin{bmatrix}
1\\
3
\end{bmatrix}
_{1,\frac{k+1-2j-2ki+2t+2s}{k\d_1}},
\end{split}
\end{gather}
which is written as
\begin{align}
\label{eq:deg42}
&
\prod_{i=1}^{\d_1}
\prod_{t=1}^{k\d_1-2}
\prod_{s=1}^{t}
\begin{bmatrix}
1\\
3
\end{bmatrix}
_{\frac{3}{k\d_1},\frac{-k\d_1+k+2-2ki+2t+2s}{k\d_1}}
\quad (k\d_1\equiv 0 \mod 3),
\\
\begin{split}
&
\prod_{i=1}^{\d_1}
\prod_{t=1}^{(k\d_1-1)/6}
\Biggl(
\begin{bmatrix}
1\\
3
\end{bmatrix}
_{\frac{1}{k\d_1},\frac{-k\d_1+k-4-2ki+12t}{k\d_1}}
\begin{bmatrix}
1\\
3
\end{bmatrix}
_{\frac{1}{k\d_1},\frac{k\d_1+k-6-2ki+12t}{k\d_1}}
\\
&
\qquad
\times
\prod_{s=1}^{k\d_1-4}
\begin{bmatrix}
1\\
3
\end{bmatrix}
_{\frac{1}{k\d_1},\frac{-k\d_1+k-2-2ki+12t+2s}{k\d_1}}
\Biggr)
\quad  (k\d_1\equiv 1 \mod 6),
\end{split}
\\
\begin{split}
&\prod_{i=1}^{\d_1}
\prod_{t=1}^{(k\d_1-2)/6}
\Biggl(
\begin{bmatrix}
1\\
3
\end{bmatrix}
_{\frac{1}{k\d_1},\frac{-k\d_1+k-4-2ki+12t}{k\d_1}}
\begin{bmatrix}
1\\
3
\end{bmatrix}
_{\frac{1}{k\d_1},\frac{k\d_1+k-4-2ki+12t}{k\d_1}}
\\
&\qquad\times
\prod_{s=1}^{k\d_1-3}
\begin{bmatrix}
1\\
3
\end{bmatrix}
_{\frac{1}{k\d_1},\frac{-k\d_1+k-2-2ki+12t+2s}{k\d_1}}
\Biggr)
\quad (k\d_1\equiv 2 \mod 6),
\end{split}
\\
&
\prod_{i=1}^{\d_1}
\prod_{t=1}^{(k\d_1-2)/2}
\prod_{s=1}^{(k\d_2 -1)/3}
\begin{bmatrix}
1\\
3
\end{bmatrix}
_{\frac{1}{k\d_1},\frac{-k\d_1+k-2-2ki+4t+6s}{k\d_1}}
\quad (k\d_1\equiv 4 \mod 6),
\\
&
\prod_{i=1}^{\d_1}
\prod_{t=1}^{(k\d_1-2)/3}
\prod_{s=1}^{(k\d_2 -1)/2}
\begin{bmatrix}
1\\
3
\end{bmatrix}
_{\frac{1}{k\d_1},\frac{-k\d_1+k-2-2ki+6t+4s}{k\d_1}}
\quad (k\d_1\equiv 5 \mod 6).
\end{align}
Here we used the following
resummation
formula for any positive integer $p$:
The Laurent polynomial
\begin{align}
&
\sum_{j=1}^{p-2}
\sum_{t=1}^{p-j-1}
\sum_{s=1}^{t}
x^{-2j+2t+2s}
=
\frac{x^{-2p+8}(1-x^{2p-4})(1-x^{2p-2})(1-x^{2p})}
{(1-x^2)(1-x^4)(1-x^6)}
\end{align}
equals to
\begin{gather}
\displaystyle
\sum_{j=1}^{p/3}
\sum_{t=1}^{p-2}
\sum_{s=1}^t
x^{4-6j+2t+2s}
\quad
(p\equiv 0 \mod 3),
\\
\displaystyle
\sum_{j=1}^{p}
\sum_{t=1}^{(p-1)/6}
x^{-2p-6+2j+12t}\biggl(1+x^{2p-2}+\sum_{s=1}^{p-4}x^{2+2s}\biggr)
\quad
(p\equiv 1 \mod 6),
\\
\displaystyle
\sum_{j=1}^{p}
\sum_{t=1}^{(p-2)/6}
x^{-2p-6+2j+12t}\biggl(1+x^{2p}+\sum_{s=1}^{p-3}x^{2+2s}\biggr)
\quad
(p\equiv 2 \mod 6),
\\
\displaystyle
\sum_{j=1}^{p}
\sum_{t=1}^{(p-2)/2}
\sum_{s=1}^{(p-1)/3}
x^{-2p-4+2j+4t+6s}
\quad
(p\equiv 4 \mod 6),
\\
\displaystyle
\sum_{j=1}^{p}
\sum_{t=1}^{(p-2)/3}
\sum_{s=1}^{(p-1)/2}
x^{-2p-4+2j+6t+4s}
\quad
(p\equiv 5 \mod 6).
\end{gather}

Again, we need to resolve the discrepancy between 
\eqref{eq:deg41} and \eqref{eq:deg42} 
for $n=(1,3)$ when $\d_1\equiv 0$ and $k\not\equiv 0$ mod 3,
e.g., $(\d_1,\d_2)=(3,6)$.
The positivity formula that is similar to \eqref{eq:deg33} exists.
However, it becomes much more complicated even for just writing it down.
First, we note that, due to the translation invariance of
the expression   \eqref{eq:deg42}  with respect to the index $i$ therein,
it is enough to concentrate on the case $\d_1=3$.
Let 
\begin{align}
\begin{split}
A_k(x)
:=&\sum_{i=1}^3 \sum_{t=1}^{3k-2}\sum_{s=1}^t
 x^{-2ki+2t+2s+6k-4}
 \\
  =&\
  \frac{(1-x^{6k-4})(1-x^{6k-2})(1-x^{6k})}
  {(1-x^2)(1-x^4)(1-x^{2k})}
  \end{split}
\end{align}
be the (normalized) generating polynomial of the quantum data of
the product  \eqref{eq:deg42} for $\d_1=3$.
Let
\begin{align}
\label{eq:Bp1}
B_k(x):=\frac{A_k(x)}{1+x^2+x^4}
=  \frac{(1-x^{6k-4})(1-x^{6k-2})(1-x^{6k})}
  {(1-x^4)(1-x^6)(1-x^{2k})}.
\end{align}
Then, the positivity property of the product  \eqref{eq:deg42} 
 is equivalent
to the fact that $B_k(x)$ is a polynomial in $x$ with nonnegative integer
coefficients for $k\not\equiv 0 $ mod 3.
This is not obvious from the expression \eqref{eq:Bp1}.
However, one can find and prove such a positive expression with the help of a
computer algebra software. (We used SageMath 9.4.)
We separate it into the cases $k\equiv 1$, 2, 4, 5 mod 6.

(i). Let $k=6p+1$ ($p\in \bbZ_{\geq 0}$). Then, the polynomial $B_k(x)$ has the following positive expression.
\begin{align*}
B_{6p+1}(x)
&=
\sum_{t=1}^p \sum_{s=1}^3
t(x^{12t+4s-16} + x^{-12t-4s + 16 +96p})
\\
&\quad+
\sum_{t=1}^p
\bigl\{
 \sum_{s=1}^2
(p+2t-1)(x^{12t+4s-16+12p} + x^{-12t-4s + 16 +84p})
\\
&\qquad\qquad+
(p+2t)(x^{12t-4+12p} + x^{-12t +4 + 84p})
\bigr\}
\\
&\quad+
\sum_{t=1}^{3p}
(3p+t) (x^{4t-4+24p} + x^{-4t+4+72p})
\\
&\quad+
\sum_{t=1}^p \sum_{s=1}^3
(6p+t) (x^{12t+4s-16+36p} + x^{-12t-4s + 16 +60p})
\\
&\quad+
(7p+1) x^{48p}
\\
&\quad+
\sum_{t=1}^{p-1} \sum_{s=1}^3
t(x^{12t+4s-10} + x^{-12t-4s + 10 +96p})
\\
&\quad+
\sum_{t=1}^{p+1}
\bigl\{
 \sum_{s=1}^2
(p+2t-2)(x^{12t+4s-22+12p} + x^{-12t-4s + 22 +84p})
\\
&\qquad\qquad+
(p+2t-1)(x^{12t-10+12p} + x^{-12t +10 + 84p})
\bigr\}
\\
&\quad+
\sum_{t=1}^{3p-2}
(3p+t+1) (x^{4t+2+24p} + x^{-4t-2+72p})
\\
&\quad+
6p(x^{36p-2}+x^{36p+2}+x^{60p-2}+x^{60p+2})
\\
&\quad+
\sum_{t=1}^{p-1} \sum_{s=1}^3
(6p+t) (x^{12t+4s-10+36p} + x^{-12t-4s + 10 +60p})
\\
&\quad+
7p(x^{48p-6}+x^{48p-2}+x^{48p+2}+x^{48p+6}).
\end{align*}

(ii). Let $k=6p+2$ ($p\in \bbZ_{\geq 0}$). Then, the polynomial $B_k(x)$ has the following positive expression.
\begin{align*}
B_{6p+2}(x)
&=
\sum_{t=1}^p \sum_{s=1}^3
t(x^{12t+4s-16} + x^{-12t-4s + 32 +96p})
\\
&\quad+
\sum_{t=1}^p
\bigl\{
 \sum_{s=1}^2
(p+2t)(x^{12t+4s-12+12p} + x^{-12t-4s + 28 +84p})
\\
&\qquad\qquad+
(p+2t-1)(x^{12t-12+12p} + x^{-12t +28 + 84p})
\bigr\}
\\
&\quad+
\sum_{t=1}^{3p+1}
(3p+t) (x^{4t-4+24p} + x^{-4t+20+72p})
\\
&\quad+
(6p+2)(x^{36p+4}+x^{36p+8}+x^{60p+8}+x^{60p+12})
\\
&\quad+
\sum_{t=1}^{p-1} \sum_{s=1}^3
(6p+t+2) (x^{12t+4s-4+36p} + x^{-12t-4s + 20 +60p})
\\
&\quad+
(7p+2) (x^{48p}+x^{48p+4}+x^{48p+8}+x^{48p+12}+x^{48p+16})
\\
&\quad+
\sum_{t=1}^{p} \sum_{s=1}^3
t(x^{12t+4s-10} + x^{-12t-4s + 26 +96p})
\\
&\quad+
\sum_{t=1}^{p}
\bigl\{
 \sum_{s=1}^2
(p+2t)(x^{12t+4s-6+12p} + x^{-12t-4s + 22 +84p})
\\
&\qquad\qquad+
(p+2t-1)(x^{12t-6+12p} + x^{-12t +22 + 84p})
\bigr\}
\\
&\quad+
\sum_{t=1}^{3p}
(3p+t) (x^{4t+2+24p} + x^{-4t+14+72p})
\\
&\quad+
\sum_{t=1}^{p} \sum_{s=1}^3
(6p+t) (x^{12t+4s-10+36p} + x^{-12t-4s + 26 +60p})
\\
&\quad+
(7p+1)(x^{48p+6}+x^{48p+10}).
\end{align*}

(iii). Let $k=6p+4$ ($p\in \bbZ_{\geq 0}$). Then, the polynomial $B_k(x)$ has the following positive expression.
\begin{align*}
B_{6p+4}(x)
&=
\sum_{t=1}^p \sum_{s=1}^3
t(x^{12t+4s-16} + x^{-12t-4s + 64 +96p})
\\
&\quad+
\sum_{t=1}^{p+1}
\bigl\{
 \sum_{s=1}^2
(p+2t-1)(x^{12t+4s-16+12p} + x^{-12t-4s + 64 +84p})
\\
&\qquad\qquad+
(p+2t)(x^{12t-4+12p} + x^{-12t +52 + 84p})
\bigr\}
\\
&\quad+
\sum_{t=1}^{3p+1}
(3p+t+2) (x^{4t+8+24p} + x^{-4t+40+72p})
\\
&\quad+
(6p+4)(x^{36p+16}+x^{36p+20}+x^{60p+28}+x^{60p+32})
\\
&\quad+
\sum_{t=1}^p \sum_{s=1}^3
(6p+t+4) (x^{12t+4s+8+36p} + x^{-12t-4s + 40 +60p})
\\
&\quad+
(7p+5) x^{48p+24}
\\
&\quad+
\sum_{t=1}^{p} \sum_{s=1}^3
t(x^{12t+4s-10} + x^{-12t-4s + 58 +96p})
\\
&\quad+
\sum_{t=1}^{p+1}
\bigl\{
 \sum_{s=1}^2
(p+2t-1)(x^{12t+4s-10+12p} + x^{-12t-4s + 58 +84p})
\\
&\qquad\qquad+
(p+2t)(x^{12t+2+12p} + x^{-12t +46 + 84p})
\bigr\}
\\
&\quad+
\sum_{t=1}^{3p}
(3p+t+2) (x^{4t+14+24p} + x^{-4t+34+72p})
\\
&\quad+
\sum_{t=1}^{p} \sum_{s=1}^3
(6p+t+2) (x^{12t+4s+2+36p} + x^{-12t-4s + 46 +60p})
\\
&\quad+
(7p+3)(x^{48p+18}+x^{48p+22}+x^{48p+26}+x^{48p+30}).
\end{align*}

(iv). Let $k=6p+5$ ($p\in \bbZ_{\geq 0}$). Then, the polynomial $B_k(x)$ has the following positive expression.
\begin{align*}
B_{6p+5}(x)
&=
\sum_{t=1}^{p+1} \sum_{s=1}^3
t(x^{12t+4s-16} + x^{-12t-4s + 80 +96p})
\\
&\quad+
\sum_{t=1}^p
\bigl\{
 \sum_{s=1}^2
(p+2t+1)(x^{12t+4s+12p} + x^{-12t-4s + 64 +84p})
\\
&\qquad\qquad+
(p+2t)(x^{12t+12p} + x^{-12t +64 + 84p})
\bigr\}
\\
&\quad+
\sum_{t=1}^{3p+3}
(3p+t+1) (x^{4t+8+24p} + x^{-4t+56+72p})
\\
&\quad+
\sum_{t=1}^p \sum_{s=1}^3
(6p+t+4) (x^{12t+4s+8+36p} + x^{-12t-4s + 56 +60p})
\\
&\quad+
(7p+5) (x^{48p+24} +x^{48p+28}+x^{48p+32}+x^{48p+36}+x^{48p+40})
\\
&\quad+
\sum_{t=1}^{p} \sum_{s=1}^3
t(x^{12t+4s-10} + x^{-12t-4s + 74 +96p})
\\
&\quad+
\sum_{t=1}^{p+1}
\bigl\{
 \sum_{s=1}^2
(p+2t)(x^{12t+4s-6+12p} + x^{-12t-4s + 70 +84p})
\\
&\qquad\qquad+
(p+2t-1)(x^{12t-6+12p} + x^{-12t +70 + 84p})
\bigr\}
\\
&\quad+
\sum_{t=1}^{3p+2}
(3p+t+2) (x^{4t+14+24p} + x^{-4t+50+72p})
\\
&\quad+
(6p+4)(x^{36p+26}+x^{60p+38})
\\
&\quad+
\sum_{t=1}^{p} \sum_{s=1}^3
(6p+t+4) (x^{12t+4s+14+36p} + x^{-12t-4s + 50 +60p})
\\
&\quad+
(7p+5) (x^{48p+30}+x^{48p+34}).
\end{align*}

\bibliography{../../biblist/biblist.bib}
\end{document}